\documentclass[11pt]{amsart}
\usepackage{amsmath,amsthm,amssymb,latexsym,color}
\usepackage{hyperref}
\usepackage{lipsum}
\usepackage{todonotes}
\usepackage{multirow}
\usepackage[mathscr]{eucal}

\usepackage{xcolor}
\date{}
\usepackage{pgfplots}
\usepackage{tikz}
\usetikzlibrary{arrows,matrix,positioning}
\usepackage{graphicx}
\usepackage{subfigure}

\newtheorem{theorem}{Theorem}[section]
\newtheorem*{theorem*}{Theorem}

\newtheorem{lemma}[theorem]{Lemma}
\newtheorem{cor}[theorem]{Corollary}
\newtheorem{defi}[theorem]{Definition}
\newtheorem{notation}[theorem]{Notation}
\newtheorem{prop}[theorem]{Proposition}

\theoremstyle{definition}
\newtheorem{Remark}[theorem]{Remark}
\newtheorem{Assumption}[theorem]{Assumption}
\newtheorem{Example}[theorem]{Example}

\theoremstyle{plain}

\newcommand{\N}{\mathbb{N}}

\newcommand{\R}{\mathbb{R}}
\newcommand{\C}{\mathbb{C}}

\newcommand{\Exp}{\mathbb{E}}

\newcommand{\Event}{\mathcal{E}}
\def\Prob{{\mathbb P}}

\def\Id{{\rm Id}}

\def\normal{{\mathfrak n}}

\def\spn{{\rm span\,}}

\def\col{{\rm col}}

\def\row{{\rm row}}

\def\Net{{\mathcal N}}

\title{On pseudospectrum of inhomogeneous non-Hermitian random matrices}
\author{
Konstantin Tikhomirov
}
\address{
\medskip
\noindent
Department of Mathematical Sciences\\
Carnegie Mellon University\\
Wean Hall 6113\\
Pittsburgh, PA 15213\\
\texttt{\small
e-mail:   ktikhomi@andrew.cmu.edu}
}

\thanks{The work is partially supported by the NSF Grant DMS 2054666}

\begin{document}

\maketitle

\begin{abstract}
Let $A$ be an $n\times n$
matrix with mutually independent centered Gaussian entries.
Define
\begin{align*}
\sigma^*:=\max\limits_{i,j\leq n}\sqrt{\Exp\,|A_{i,j}|^2},\quad
\sigma:=\max\bigg(\max\limits_{j\leq n}\sqrt{\Exp\,\|\col_j(A)\|_2^2},
\max\limits_{i\leq n}\sqrt{\Exp\,\|\row_i(A)\|_2^2}\bigg).
\end{align*}
Assume that $\sigma\geq n^\varepsilon\,\sigma^*$ for a constant $\varepsilon>0$,
and that a complex number $z$ satisfies
$|z|=\Omega(\sigma)$. We prove that
$$
s_{\min}(A-z\,\Id)
\geq |z|\,\exp\bigg(-n^{o(1)}\,\Big(\frac{\sqrt{n}\,\sigma^*}{\sigma}\Big)^2\bigg)
$$
with probability $1-o(1)$.
Without extra assumptions on $A$, the bound is optimal
up to the $n^{o(1)}$ multiple in the power of exponent. 
We discuss applications of this estimate in context
of empirical spectral distributions of inhomogeneous non-Hermitian random matrices.
\end{abstract}


\section{Introduction}

For each $n\geq 1$, let $A_n$ be an $n\times n$ random matrix.
Denote by $\mu_n$ its empirical spectral distribution i.e a discrete probability measure
$$
\mu_n=\frac{1}{n}\sum_{i=1}^n \delta_{\lambda_i^{(n)}},
$$
where $\delta_x$ is the Dirac delta function with mass at $x$, and $\lambda_1^{(n)},\dots,\lambda_n^{(n)}$
are eigenvalues of $A_n$ arranged arbitrarily.

In the setting where each matrix $A_n$ has i.i.d entries,
study of the limiting behavior of $(\mu_n)_{n=1}^\infty$
is a well developed line of research within the random matrix theory.
In particular, if for every $n$ the entries of $A_n$ have variances $1/n$
and under very mild additional assumptions, the sequence $(\mu_n)_{n=1}^\infty$
is known to follow the {\it circular law} i.e converges weakly to
the uniform measure on the unit disc of the complex plane;
see \cite{B1997,BR2019,Ginibre,Girko,GT2010,PZ2010,RT2019,TV2008,TV2010,W2012}
as well as survey \cite{BC2012} and PhD thesis \cite{Cook2016} for more information.
We note here that
rescaled adjacency matrices of directed $d$--regular random graphs (for
$d=d(n)=\omega(1)$)
follow the same global limiting law as was established in \cite{Cook2019,BCZ,LLTTY}.

Matrices encountered in applications are often {\it structured},
with locations of zeros and relative magnitudes of the entries
determined by the nature of a problem (see, in particular, \cite{Olshevsky}).
Random matrices with identically distributed entries cannot serve as an
adequate model of a structured matrix, and, as a natural generalization step,
one can consider {\it inhomogeneous} non-Hermitian random matrices with mutually independent entries
having different variances and, in particular, allowing zero entries.
The spectral norm of such matrices, as well as their Hermitian counterparts, has been actively studied, see, in particular,
\cite{Latala,RS2013,BvH,vH17,BBvH,BvH22}.
On the other hand, the spectral distribution
is much less understood and, as of this writing,
has been investigated only in
specific settings. We refer to \cite{TV2010,CHNR1,CHNR2,JS21,AEK2018,AEK2021} for the study of structured dense matrices
(in particular, \cite{AEK2018,AEK2021} for local spectral distribution and spectral radius of inhomogeneous matrices with entries of comparable magnitudes)
and \cite{JJLO} for block-band random matrices.

\medskip

As of now, the mainstream approach to the study of 
spectrum of non-Hermitian matrices, developed by Girko \cite{Girko},
is a {\it Hermitization argument} in which the empirical spectral distribution of $A_n$
is related to the singular spectrum $s_1(A_n-z\,\Id)\geq s_2(A_n-z\,\Id)\geq\dots
\geq s_n(A_n-z\,\Id)$ of shifted matrices $A_n-z\,\Id$, $z\in\C$ via
the formula
\begin{equation}\label{akfgvytfvuqtcviyvc}
\prod_{i=1}^n |\lambda_i^{(n)}-z|=\big|\det(A_n-z\,\Id)\big|=\prod_{i=1}^n s_i(A_n-z\,\Id),
\end{equation}
so that a weak limit for the sequence $(\mu_n)_{n=1}^\infty$
(if it exists) can be identified by estimating the logarithm
of the right hand side of \eqref{akfgvytfvuqtcviyvc} for almost every $z\in\C$
(see \cite{BC2012,TV2010} for more information).
The latter, in turn, is usually split into two subproblems:
\begin{itemize}
\item[(I)] Computing a limit for sequence of random measures $\frac{1}{n}\sum_{i=1}^n \delta_{s_i^2(A_n-z\,\Id)}$, $n\geq 1$.
\item[(II)] Proving that 
$s_{\min}(A_n-z\,\Id)=s_n(A_n-z\,\Id)\geq \varepsilon_n$, for a specific choice of $\varepsilon_n>0$,
with probability $1-o(1)$.
\end{itemize}
Problems (I) and (II) are of different nature.
For i.i.d and certain structured models, (I)
was successfully addressed by studying the Stieltjes
transform of the singular spectrum; we refer, in particular, to \cite[Chapter~11]{BS2010}
and survey \cite{BC2012} as well as recent research works
\cite{CHNR1,JJLO,BBvH} for details and further references.

Problem (II) is often approached with ``geometric''
methods based on evaluating distances between random vectors
and random subspaces associated with a random matrix.
There is substantial literature on the subject, and we refer to survey \cite{ICM} for more information
and further references.
However, most of the effort has been focused on matrices
with identically distributed entries, and results dealing with inhomogeneous matrices are scarce.
As an example of the latter, in \cite{RZ,CookInv} bounds
on $s_{\min}(A_n-z\,\Id)$ were obtained for matrices with
{\it broad connectivity} and {\it robust irreducibility} properties
which can be seen as relatives
of graph expansion.
Matrices with non-identically distributed entries of comparable magnitudes
were considered, in particular, in \cite{Galya,LTV,JS21}.
An argument based on the regularity lemma was developed in
\cite{CookInv} to deal with arbitrary structured random matrices provided that
the number of non-zero entries is a constant proportion of $n^2$.
Block-band matrices were recently considered in \cite{JJLO}.
Without the specific assumptions on locations of non-zero entries
or matrix density, quantifying invertibility of inhomogeneous matrices
has remained an open problem. 

\bigskip

The goal of this paper is to make progress on problem (II) in a general setting
of inhomogeneous matrices with mutually independent entires, without
any specific restrictions on the variance profile. The main result is

\begin{theorem}\label{alryywtvcutwercuqyt}
For every $R\geq 1$ and $\varepsilon\in(0,1]$
there is $n_0\in\N$
depending on $R$ and $\varepsilon$ with the following property.
Let $n\geq n_0$,
and let $A$ be an $n\times n$ non-zero random matrix
with mutually independent centered real Gaussian entries.
Define
$$
\sigma^*:=\max\limits_{i,j\leq n}\sqrt{\Exp\,|A_{i,j}|^2},\quad
\sigma:=\max\bigg(\max\limits_{j\leq n}\sqrt{\Exp\,\|\col_j(A)\|_2^2},
\max\limits_{i\leq n}\sqrt{\Exp\,\|\row_i(A)\|_2^2}\bigg).
$$
Assume additionally that a complex number $z$ satisfies
$$
|z|\geq \max\Big(\sigma^*\,n^{\varepsilon},\frac{\sigma}{R}\Big).
$$
Then with probability at least $1-\frac{1}{n}$ we have
$$
s_{\min}(A-z\,\Id)\geq |z|\,\exp\bigg(-n^{\varepsilon}\,\Big(\frac{\sqrt{n}\,\sigma^*}{\sigma}\Big)^2\bigg).
$$
\end{theorem}
In fact, we prove a more general result for random matrices with independent subgaussian
entries with bounded distribution densities (see Theorem~\ref{qiuyefqiwgfqvuq}).
The results
can be viewed as statements about coverage functions of the pseudospectrum of 
random inhomogeneous matrices.
Specifically, fix a constant $\varepsilon>0$, and for each $n\geq 1$ let
$A_n$ be an $n\times n$ matrix with independent centered real Gaussian entries such that
$$
\sigma_n^*=\max\limits_{i,j\leq n}\sqrt{\Exp\,\big|(A_n)_{i,j}\big|^2}\quad\mbox{ and }\quad
\sigma_n=\max\bigg(\max\limits_{j\leq n}\sqrt{\Exp\,\|\col_j(A_n)\|_2^2},
\max\limits_{i\leq n}\sqrt{\Exp\,\|\row_i(A_n)\|_2^2}\bigg)
$$
satisfy $n^\varepsilon\,\sigma^*_n\leq \sigma_n=O(1)$.
Further, set $\delta_n:=\exp\big(-n^{\varepsilon}\,\big(\frac{\sqrt{n}\,\sigma_n^*}{\sigma_n}\big)^2\big)$,
$n\geq 1$.
The {\it $\delta_n$--pseudo\-spectrum} of the random matrix $A_n$ is defined as a random set
$$
\Lambda_{\delta_n}(A_n):=\big\{\lambda\in \C:\;s_{\min}(A_n-\lambda\,\Id)\leq \delta_n\big\},
$$
and the {\it one-point coverage function} $p_{\Lambda_{\delta_n}(A_n)}:\C\to[0,1]$
is given by
$$
p_{\Lambda_{\delta_n}(A_n)}(w):=\Prob\big\{w\in \Lambda_{\delta_n}(A_n)\big\},\quad w\in\C
$$
(see, for example, \cite[p.~33]{Molchanov}).
Theorem~\ref{alryywtvcutwercuqyt} then implies that for every non-zero $z\in\C$,
$p_{\Lambda_{\delta_n}(A_n)}(z)=o(1)$.


\medskip

The ratio $\frac{\sqrt{n}\,\sigma^*}{\sigma}$ in Theorem~\ref{alryywtvcutwercuqyt}
should be interpreted as a measure
of sparsity; for example, if $A$ is a standard $n\times n$ Gaussian matrix then
$\frac{\sqrt{n}\,\sigma^*}{\sigma}=1$ whereas for a diagonal matrix with standard Gaussians
on the diagonal, $\frac{\sqrt{n}\,\sigma^*}{\sigma}=\sqrt{n}$.
As a more general example, fix a constant $\varepsilon\in(0,1]$ and for each $n$
let $d_n$ be an integer in the interval $[n^\varepsilon,n]$
and let $V_n$ be an $n\times n$ deterministic $0/1$--matrix with at most $d_n$
ones in every row and column. Define $A_n$ as the Hadamard (entry-wise)
product of $\frac{1}{\sqrt{d_n}}V_n$ with a standard $n\times n$ real Gaussian matrix.
Theorem~\ref{alryywtvcutwercuqyt} then implies
that for every non-zero complex number $z$,
$$
\Prob\bigg\{
s_{\min}(A_n-z\,\Id)\geq \exp\bigg(-n^{o(1)}\;\frac{n}{d_n}\bigg)
\bigg\}=1-o(1).
$$
Without any extra assumptions on the variance profiles $V_n$, this lower bound
is best possible up to the multiple $n^{o(1)}$ in the exponent.
We refer to the beginning
of Section~\ref{asytfvufytviyvbbntr} for details.

To our best knowledge, Theorem~\ref{alryywtvcutwercuqyt}
is the first result in literature which provides quantitative
bounds on the smallest singular value of inhomogeneous matrices
without special assumptions on the structure
(such as expansion-like properties or decomposition into ``homogeneous'' blocks)
or matrix density\footnote{The drawback of this generality is suboptimal
estimates for some specific choices of variance profile. In particular, we conjecture
that in the setting of doubly stochastic profiles i.e under the extra assumptions that
$\Exp\,\|\col_j(A)\|_2^2=1$, $j\leq n$, $\Exp\,\|\row_i(A)\|_2^2=1$, $i\leq n$,
the lower bound on $s_{\min}(A-z\,\Id)$ can be significantly improved.}.
We expect that it will lead to new results on the spectrum of non-Hermitian
structured random matrices, assuming complementary advances on subproblem (I)
from our earlier discussion.

As an illustration, we consider the well known problem of identifying
the limiting spectral distribution of non-Hermitian random periodic band matrices.
For each $n$,
assume that $B_n$ is an $n\times n$ matrix where the entries are mutually independent
and the $(i,j)$--th entry is standard Gaussian if and only if $(i-j) \mod n\leq w_n$ or $(j-i) \mod n\leq w_n$
(and all other entries are zeros). The parameter $w_n\leq n/2$ is the bandwidth.
In the setting where $\frac{w_n}{n}=\Omega(1)$, works \cite{CookInv,CHNR1} imply
the circular law for the sequence of
empirical spectral distributions of matrices $\frac{1}{\sqrt{2w_n+1}}B_n$.
However, the case of power-law decay, with $\frac{w_n}{n}\leq n^{-\varepsilon}$ for a fixed $\varepsilon>0$,
has not been covered by any existing results. 
In Section~\ref{fiauftucytvuyrbiuehbiu} of this paper, we apply Theorem~\ref{alryywtvcutwercuqyt}
together with results from \cite{JS,JJLO} to derive
\begin{cor}[Circular law for periodic band matrices]\label{ajhfbaifhbijvnjgvaigvi}
There is a universal constant $c>0$ with the following property.
Let $(w_n)_{n\geq 1}$ be a sequence of integers where for each large $n$, $w_n$ satisfies
$$
n^{33/34}\leq w_n< cn.
$$
Then the sequence of
empirical spectral distributions of matrices $\frac{1}{\sqrt{2w_n+1}}B_n$
converges weakly in probability to
the uniform measure on the unit disc of the complex plane.
\end{cor}

\bigskip

As the final part of the introduction, we give a high-level
overview of the proof of the main result.
To be able to present the idea concisely, we will hide certain technical details;
for that reason the outline below should not be viewed as a perfectly fair description
of the actual proof.
The basic principle is, for
every point of the event ``$s_{\min}(A-z\,\Id)$ is small'', to identify
a nested sequence of submatrices of $A-z\,\Id$
satisfying some rare properties. The condition that the properties are unlikely
implies that the event ``$s_{\min}(A-z\,\Id)$ is small''
has small probability.
To be more specific, for every subset $J\subset[n]$ let $A_J$ be the $J\times J$
principal submatrix of $A$. Then, at every point of the event in question, we find a number $d\geq 0$,
a sequence of subsets $[n]=:J_0\supset J_1\supset J_2\dots\supset J_d$,
and indices $j_0\in J_0\setminus J_1$, $j_1\in J_1\setminus J_2$, $\dots$, $j_{d-1}
\in J_{d-1}\setminus J_d$ such that for each $\ell\in\{0,1,\dots,d-1\}$,
the distance from the column $\col_{j_\ell}(A_{J_\ell}-z\,\Id)$
to the span of columns $E_\ell:=\spn\big\{\col_{k}(A_{J_\ell}-z\,\Id), \;k\in J_\ell\setminus\{j_\ell\}\big\}$
is much less than the ``typical'' distance 
guaranteed by basic anti-concentration
estimates for linear combinations of random variables with bounded distribution densities
(see Lemma~\ref{auevfiwuyfviygvcikygviygv} below).
The mutual independence of the entries of $A$ and a telescopic
conditioning argument imply that the probability of the event
``$s_{\min}(A-z\,\Id)$ is small'' can be bounded by the product of the probabilities
that $\col_{j_\ell}(A_{J_\ell}-z\,\Id)$ is close to $E_\ell$, $\ell=0,\dots,d-1$.
The actual situation is more involved since the nested sequence $[n]=:J_0\supset J_1\supset J_2\dots\supset J_d$
and the column indices are random, introducing complicated dependencies in matrices $A_{J_\ell}$.
As a natural decoupling argument, we compute the product of probabilities
for every admissible deterministic choice of the sets $J_\ell$ and indices $j_\ell$, $\ell\geq 0$,
and take the union probability bound over the admissible choices
as an upper bound for $\Prob\{s_{\min}(A-z\,\Id)\mbox{ is small}\}$.
The term ``admissible'' here does not mean any of superexponentially many nested sequences
and instead is determined by structure of supports of the matrix columns.
As a way to define and navigate through the admissible choices of the subsets and indices,
we introduce a directed graph whose vertices are principal submatrices
of $A-z\,\Id$, so that each admissible sequence $(J_\ell)_{\ell=0}^d$
is associated with a path on that graph (the graph is constructed in Subsection~\ref{alrberufvifyviyqviv}).
Each admissible path ends either at an ``empty'' matrix or at a submatrix of $A-z\,\Id$
with certain conditions on expected squared norms of its columns
(we call the latter {\it non-empty terminals}).
An important step in estimating $\Prob\{s_{\min}(A-z\,\Id)\mbox{ is small}\}$
is a uniform quantitative control of invertibility of the non-empty terminals.
That task is split into three substeps: certain Gershgorin--type
estimate for block matrices (Subsection~\ref{tkhjniuefuytvcutv}),
a uniform bound on norms of submatrices of $A$ (Subsection~\ref{aagvcuytwevutwcvrg}), and 
construction of a special block decomposition of the terminals (Subsection~\ref{aefiwtvcytarcytefufwa}).


\bigskip

{\bf{}Acknowledgment.}
The author is grateful to Han Huang for helpful discussions.

\section{Notation and preliminaries}\label{alueyfgiuyiuwebubfh}

\subsection{Notation and definitions}

\begin{notation}[Submatrices]
Given an $n\times n$ matrix $M$ and non-empty subsets $I,J\subset[n]$,
denote by $M_{I,J}$ the submatrix of $M$ obtained by removing rows indexed over $I^c$ 
and columns indexed over $J^c$. We will assume that the entries of $M_{I,J}$
are indexed over the product $I\times J$.
When $I=J$ we will write $M_I$ instead of $M_{I,I}$.
\end{notation}

\begin{notation}[Empty matrix]
We will write $[\;]$ to denote an ``empty'' zero by zero matrix.
\end{notation}

\begin{notation}[Identity matrix]
We will denote identity matrices by $\Id$. The matrix dimensions shall always be clear from context.
\end{notation}

\begin{notation}[Matrix norms]
Given a $k\times\ell$ matrix $M$, we denote by $\|M\|_{HS}$ its Hilbert--Schmidt norm.
Further, we define $\|M\|_{\infty\to 2}$ as the supremum of $\|Mx\|_2$ over all vectors $x\in\C^\ell$
with unit $\|\cdot\|_\infty$--norm.
\end{notation}

\begin{notation}[Product of vectors]
Given two $m$--dimensional complex vectors $x=(x_1,\dots,x_m)$ and $y=(y_1,\dots,y_m)$,
we denote by $\langle x,y\rangle$ the sum
$$
\langle x,y\rangle=\sum_{i=1}^m x_i\,y_i.
$$
Note that this definition differs from that of the standard inner product of two vectors $x,y$ in $\C^m$, 
$\sum_{i=1}^m x_i\,y_i^*$.
\end{notation}

\begin{notation}[Concatenation of vectors]
Given vectors $x^{(1)}\in\C^{m_1}$, $x^{(2)}\in\C^{m_2}$, $\dots$, $x^{(k)}\in\C^{m_k}$,
denote by $\oplus_{j\leq k}\; x^{(j)}=x^{(1)}\oplus x^{(2)}\oplus\dots\oplus x^{(k)}$
their concatenation.
\end{notation}

\begin{notation}[Vectors in $\C^I$]
Given a non-empty finite set $I$, denote by $\C^I$ the complex space
of $|I|$--dimensional vectors with vector components indexed over $I$.
\end{notation}

\begin{defi}[Subgaussian variables and subgaussian norm]
Let $K\geq 1$. We say that a real or complex variable $\xi$ is {\bf $K$--subgaussian}
if
$$
\Exp\,\exp(|\xi|^2/K^2)\leq 2.
$$
We will denote the smallest $K$ such that $\xi$ is $K$--subgaussian, by $\|\xi\|_{\psi_2}$.
\end{defi}

\subsection{Subgaussian concentration}\label{aljrhgbeiuavucvutv}

\begin{theorem}[{See, for example, \cite[Section~2.6]{VershyninBook}}]\label{akjfgifyutviywvgcywv}
There is a universal constant $C_{\text{\tiny\ref{akjfgifyutviywvgcywv}}}>0$
with the following property.
Let $X_1,\dots,X_m$ be mutually independent subgaussian random variables,
and let $a_1,\dots,a_m\in\C$ be any scalars. Then the linear combination $\sum_{i=1}^m a_i\,X_i$
is subgaussian, with
$$
\Big\|
\sum_{i=1}^m a_i\,X_i
\Big\|_{\psi_2}^2
\leq C_{\text{\tiny\ref{akjfgifyutviywvgcywv}}}\,
\sum_{i=1}^m |a_i|^2\,\|X_i\|_{\psi_2}^2.
$$
\end{theorem}

\begin{theorem}[{See, for example, \cite[Section~3.1]{VershyninBook}}]\label{aiuyfgquytfeuqytvcytwqv}
For every $K\geq 1$ there are $C_{\text{\tiny\ref{aiuyfgquytfeuqytvcytwqv}}},
c_{\text{\tiny\ref{aiuyfgquytfeuqytvcytwqv}}}>0$ depending on $K$
with the following property.
Let
$X=(X_1,\dots,X_m)$ be a vector of mutually independent $K$--subgaussian
random variables.
Then the Euclidean norm $\|X\|_2$ satisfies
$$
\Prob\big\{\|X\|_2\geq C_{\text{\tiny\ref{aiuyfgquytfeuqytvcytwqv}}}\,\sqrt{n}+t\big\}
\leq 2\exp\big(-c_{\text{\tiny\ref{aiuyfgquytfeuqytvcytwqv}}}\,t^2\big),\quad t\geq 0.
$$
\end{theorem}

\subsection{Anti-concentration of combinations of r.v.'s with bounded distribution densities}

\begin{lemma}\label{auevfiwuyfviygvcikygviygv}
There is a universal constant $C_{\text{\tiny\ref{auevfiwuyfviygvcikygviygv}}}>0$ with the following property.
Let $\rho_0>0$ be a parameter.
Let $y\in\C^m$ be a complex vector, and let $X=(X_1,\dots,X_m)$
be a random vector with $m$ mutually independent
centered components of finite absolute second moments satisfying one of the following:
\begin{itemize}
\item Either all components of $X$ are real and the distribution
densities of normalized variables $\frac{X_i}{\sqrt{\Exp\,|X_i|^2}}$, $i\leq m$, are
uniformly bounded above by $\rho_0$,
\item Or all components of $X$ are complex, with independent
real and imaginary parts, and the distribution densities of
$\frac{\Re X_i}{\sqrt{\Exp\,|X_i|^2}}$ and $\frac{\Im X_i}{\sqrt{\Exp\,|X_i|^2}}$, $i\leq m$,
are uniformly bounded above by $\rho_0$.
\end{itemize}
Then for any complex number $s\in \C$ we have
$$
\Prob\bigg\{\bigg|\frac{\langle y,X\rangle}{\sqrt{\sum_{i=1}^m |y_i|^2\,\Exp\,|X_i|^2}}-s\bigg|
\leq t\bigg\}\leq C_{\text{\tiny\ref{auevfiwuyfviygvcikygviygv}}}\,\rho_0\,t,\quad t>0.
$$
\end{lemma}
\begin{proof}
Note that non-random vector
$$
v=\bigg(\frac{y_i\cdot \sqrt{\Exp\,|X_i|^2}}{\sqrt{\sum_{i=1}^m |y_i|^2\,\Exp\,|X_i|^2}}\bigg)_{i=1}^m
$$
has unit Euclidean norm.
Without loss of generality (by multiplying $y$ by ${\bf i}$ if necessary),
we can assume that $\|\Re v\|_2\geq 1/2$.
We write
$$
\frac{\langle y,X\rangle}{\sqrt{\sum_{i=1}^m |y_i|^2\,\Exp\,|X_i|^2}}
=\sum_{i=1}^m \frac{X_i}{\sqrt{\Exp\,|X_i|^2}}\cdot
\frac{y_i\cdot \sqrt{\Exp\,|X_i|^2}}{\sqrt{\sum_{i=1}^m |y_i|^2\,\Exp\,|X_i|^2}}.
$$
Fix any realization of imaginary parts of $X_i$'s (if any). Then we can write
$$
\frac{\Re\langle y,X\rangle}{\sqrt{\sum_{i=1}^m |y_i|^2\,\Exp\,|X_i|^2}}
=\sum_{i=1}^m \frac{\Re X_i}{\sqrt{\Exp\,|X_i|^2}}\cdot \Re v_i+\xi,
$$
where $\xi$ is some non-random number.
Applying well known anti-concentration results for linear combinations
of random variables with continuous distributions (see \cite{Rogozin1987} as well as
\cite[Theorem~1.2]{RV15}),
we get that the distribution density of
$$
\sum_{i=1}^m \frac{\Re X_i}{\sqrt{\Exp\,|X_i|^2}}\cdot \Re v_i
$$
is bounded above by a constant multiple of $\rho_0$.
The result follows.
\end{proof}

\subsection{A block-matrix Gershgorin--type estimate}\label{tkhjniuefuytvcutv}

The standard proof of Gershgorin's circle theorem 
can be easily adapted to estimate the smallest singular value
of diagonally dominant matrices. In this subsection, we consider a result
in the same spirit, for block matrices 
satisfying a kind of ``upper triangular domination''.
Statements of that type may be known but we are not aware of a reference.

\begin{prop}[A Gershgorin--type estimate for block matrices]\label{aljhfbiuyebfiufybeffasf}
Let $k\geq 1$ and assume that parameters $R\geq 1$ and $\varepsilon\in(0,1]$
satisfy $(4R)^{k}\varepsilon k\leq 1/2$.
Let $T$ be a block matrix of the form
$$
T=
\begin{pmatrix}
T_{1,1} & T_{1,2} & \dots & T_{1,k} \\
T_{2,1} & T_{2,2} & \dots & T_{2,k} \\
\dots & \dots & \dots & \dots \\
T_{k,1} & T_{k,2} & \dots & T_{k,k}
\end{pmatrix},
$$
where the diagonal blocks $T_{i,i}$, $i\leq k$, are square identity matrices,
$\|T_{i,j}\|\leq R$ for all $j>i$ and $\|T_{i,j}\|\leq \varepsilon$ for all $j<i$.
Then $s_{\min}(T)\geq \varepsilon k$.
\end{prop}
\begin{proof}
We will prove the statement by contradiction. Assume that there is a unit
vector $x$ such that $\|Tx\|_2\leq \varepsilon k$. We represent $x$ as a concatenation
$$
x=\oplus_{i\leq k}\;x_i,
$$
where for each $i\leq k$ the vector $x_i$ and the matrix $T_{i,i}$ have compatible dimensions.
Next, we show by induction that the components $x_i$ satisfy
\begin{equation}\label{auyvfiytfvfvcutwrcwutq}
\|x_i\|_2\leq (4R)^{k-i+1}\varepsilon k,\quad i\leq k,
\end{equation}
which will lead to contradiction since, in view of our assumptions on parameters,
$$
\sum_{i=1}^k (4R)^{k-i+1}\varepsilon k<1.
$$
For $i=k$, we have from the assumptions on $\|Tx\|_2$:
$$
\|x_k\|_2=\|T_{k,k}x_k\|_2\leq\Big\|\sum_{j=1}^{k-1}T_{k,j}x_j\Big\|_2+\varepsilon k
\leq 2\varepsilon k.
$$
At $\ell$--th step, $\ell>1$, we
assume that the inequalities in \eqref{auyvfiytfvfvcutwrcwutq} are true for $i=k-\ell+2,\dots,k$.
We then have
\begin{align*}
\|x_{k-\ell+1}\|_2&=\|T_{k-\ell+1,k-\ell+1}\,x_{k-\ell+1}\|_2\\
&\leq\Big\|\sum_{j< k-\ell+1}T_{k-\ell+1,j}\,x_j\Big\|_2+
\Big\|\sum_{j> k-\ell+1}T_{k-\ell+1,j}\,x_j\Big\|_2
+\varepsilon k\\
&\leq 2\varepsilon k+\sum_{j> k-\ell+1}R\cdot (4R)^{k-j+1}\varepsilon k\\
&\leq
2\varepsilon k+\frac{1}{3}(4R)^{\ell}\varepsilon k
\leq (4R)^{\ell}\varepsilon k,
\end{align*}
and \eqref{auyvfiytfvfvcutwrcwutq} follows.
\end{proof}

\begin{cor}\label{iuygfqwcyvqtrqvieyrc}
Let $k\geq 1$, $\varepsilon\in(0,1]$ and $R\geq 1$ satisfy
$2(8R)^{k}\varepsilon k\leq 1/2$.
Let $m\geq 1$ and let $X$ be
an $m\times m$ block matrix
with $k^2$ blocks having the form
$$
X
=\begin{pmatrix}
X_{1,1} & X_{1,2} & \dots & X_{1,k} \\
X_{2,1} & X_{2,2} & \dots & X_{2,k} \\
\dots & \dots & \dots & \dots \\
X_{k,1} & X_{k,2} & \dots & X_{k,k}
\end{pmatrix}.
$$
Let a non-zero complex number $z$ satisfy the following conditions:
\begin{itemize}
\item For all $1\leq j\leq i\leq k$, $\|X_{i,j}\|\leq\varepsilon |z|$;
\item For all $1\leq i<j\leq k$, $\|X_{i,j}\|\leq R |z|$.
\end{itemize}
Then
$$
s_{\min}(X-z\,\Id)\geq \varepsilon\,k\,|z|.
$$
\end{cor}
\begin{proof}
First, observe that the block diagonal matrix
$$
\tilde X:=
\begin{pmatrix}
X_{1,1}-z\,\Id & 0 & \dots & 0 \\
0 & X_{2,2}-z\,\Id & \dots & 0 \\
\dots & \dots & \dots & \dots \\
0 & 0 & \dots & X_{k,k}-z\,\Id
\end{pmatrix}
$$
is invertible, with
$$
s_{\min}(\tilde X)\geq |z|-\varepsilon\,|z|\geq |z|/2.
$$
Next, we estimate the smallest singular value of the matrix ${\tilde X}^{-1}\,(X-z\,\Id)$.
Note that
$$
{\tilde X}^{-1}\,(X-z\,\Id)
=
\begin{pmatrix}
\Id & (X_{1,1}-z\,\Id)^{-1}X_{1,2} & \dots & (X_{1,1}-z\,\Id)^{-1}X_{1,k} \\
(X_{2,2}-z\,\Id)^{-1} X_{2,1} & \Id & \dots & (X_{2,2}-z\,\Id)^{-1}X_{2,k} \\
\dots & \dots & \dots & \dots \\
(X_{k,k}-z\,\Id)^{-1} X_{k,1} & (X_{k,k}-z\,\Id)^{-1} X_{k,2} & \dots & \Id
\end{pmatrix},
$$
where for all $1\leq i\leq j\leq k$,
$$
\big\|(X_{i,i}-z\,\Id)^{-1} X_{i,j}\big\|\leq \frac{2}{|z|}\cdot \varepsilon\,|z|=2\varepsilon,
$$
and for all $1\leq j<i\leq k$,
$$
\big\|(X_{i,i}-z\,\Id)^{-1} X_{i,j}\big\|\leq \frac{2}{|z|}\cdot R\,|z|=2R.
$$
Applying Proposition~\ref{aljhfbiuyebfiufybeffasf} with parameters $2R$ and $2\varepsilon$, we get
$$
s_{\min}\big({\tilde X}^{-1}\,(X-z\,\Id)\big)\geq 2\varepsilon\,k,
$$
whence
$s_{\min}(X-z\,\Id)\geq \varepsilon\,k\,|z|$.
\end{proof}

\subsection{A uniform upper bound on norms of submatrices}\label{aagvcuytwevutwcvrg}

Given a square random matrix with mutually independent centered subgaussian
entries, our goal is to obtain uniform bounds on spectral
norms of its submatrices.
\begin{prop}\label{itfufytcuqtrcutqwxi}
For every $K\geq 1$ there is $C_{\text{\tiny\ref{itfufytcuqtrcutqwxi}}}>0$
depending on $K$ with the following property.
Let $n\geq 2$ and let $A$ be an $n\times n$ random matrix
such that for every entry $A_{i,j}$ with $\Prob\{A_{i,j}\neq 0\}>0$,
the normalized variable $A_{i,j}/\sqrt{\Exp\,|A_{i,j}|^2}$ is $K$--subgaussian.
Define an event
\begin{align*}
\Event_{\text{\tiny\ref{itfufytcuqtrcutqwxi}}}:=
\Big\{
&\mbox{For every choice of non-empty subsets $I,J\subset[n]$,}\\
&\|A_{I,J}\|\leq C_{\text{\tiny\ref{itfufytcuqtrcutqwxi}}}
\,\max\Big(\max\limits_{i\in I}\sqrt{\Exp\,\big\|\row_i(A_{I,J})\big\|_2^2},
\max\limits_{j\in J}\sqrt{\Exp\,\big\|\col_j(A_{I,J})\big\|_2^2}\Big)
\,\log^{5/2} n
\Big\}.
\end{align*}
Then $\Prob(\Event_{\text{\tiny\ref{itfufytcuqtrcutqwxi}}})\geq 1-n^{-2}$.
\end{prop}
\begin{Remark}[Related results in the literature]
See works
\cite{BvH,vH17,LvHY,BBvH,BvH22} 
for concentration of
norms of inhomogeneous random matrices with subgaussian entries.
Whereas for fixed subsets $I,J$ the cited works provide stronger estimates on $\|A_{I,J}\|$
than the above proposition, 
just using those estimates
as a black box and taking the union bound over $I,J\subset[n]$ would result in a highly suboptimal
statement. We did not attempt to optimize the polylogarithmic multiple $\log^{5/2}n$
in the definition of event $\Event_{\text{\tiny\ref{itfufytcuqtrcutqwxi}}}$
since it does not affect the main result of the paper.
\end{Remark}

\begin{lemma}\label{aiueyfgiwyetfciavvbva}
For every $K\geq 1$ there is $C_{\text{\tiny\ref{aiueyfgiwyetfciavvbva}}}>0$
depending on $K$ with the following property.
Let $n\geq 2$ and let $A$ be an $n\times n$ random matrix
with independent centered entries,
such that for every entry $A_{i,j}$ with $\Prob\{A_{i,j}\neq 0\}>0$,
the normalized variable $A_{i,j}/\sqrt{\Exp\,|A_{i,j}|^2}$ is $K$--subgaussian.
Define the event
\begin{align*}
\Event_{\text{\tiny\ref{aiueyfgiwyetfciavvbva}}}:=
\Big\{&\mbox{For every choice of non-empty subsets $I,J\subset[n]$,}\\
&\|A_{I,J}\|_{\infty\to 2}\leq C_{\text{\tiny\ref{aiueyfgiwyetfciavvbva}}}
\,\sqrt{\max(|I|,|J|)\log n}
\cdot
\sqrt{\max\limits_{i\in I}\,\Exp\,\big\|\row_i(A_{I,J})\big\|_2^2}
\Big\}.
\end{align*}
Then $\Prob(\Event_{\text{\tiny\ref{aiueyfgiwyetfciavvbva}}})\geq 1-n^{-2}$.
\end{lemma}
\begin{proof}
Fix for a moment any $k,\ell\in\{1,\dots,n\}$ and non-empty subsets $I,J$ of $[n]$
with $|I|=k$, $|J|=\ell$, and define
$$
a:=\sqrt{\max\limits_{i\in I}\,\Exp\,\big\|\row_i(A_{I,J})\big\|_2^2}.
$$
Let $x\in\C^J$ be any fixed vector with $|x_j|\leq 1$, $j\in J$.
By Theorem~\ref{akjfgifyutviywvgcywv}, the variables
$$
\langle \row_i(A_{I,J}),x\rangle,\quad i\in I,
$$
are mutually independent and $Ca$--subgaussian, for some $C>0$ depending on $K$.
Applying Theorem~\ref{aiuyfgquytfeuqytvcytwqv}, we obtain
$$
\Prob\big\{\big\|A_{I,J} \,x\big\|_2\geq \tilde C\,a\,\sqrt{k}+a\,t\big\}\leq 2\exp(-\tilde ct^2),\quad t>0,
$$
for some $\tilde c,\tilde C>0$ depending on $K$.
Note that for any $1/2$--net $\Net$ in the unit ball of the $\|\cdot\|_\infty$--norm in $\C^J$,
the triangle inequality for norms implies
$$
\big\|A_{I,J} \big\|_{\infty\to 2}\leq 2\max\limits_{x\in\Net}\big\|A_{I,J}\,x\big\|_2.
$$
A standard volumetric argument for estimating the size of economical nets
(see, for example, \cite[Section~4.2]{VershyninBook}) then gives
$$
\Prob\big\{\big\|A_{I,J} \big\|_{\infty\to 2}\geq 2\tilde C\,a\,\sqrt{k}+2a\,t\big\}\leq 2\hat C^{\ell}\exp(-\tilde ct^2),\quad t>0,
$$
for some universal constant $\hat C\geq 1$.
Letting $t$ be a sufficiently large multiple of $\sqrt{\max(k,\ell)\log n}$ and taking the union bound over all
${n\choose k}\cdot{n\choose \ell}$ admissible choices of $I,J$ and all $k,\ell$,
we get the result.
\end{proof}

\begin{lemma}\label{ajhgbvoinpti}
For every $K\geq 1$ there is $C_{\text{\tiny\ref{ajhgbvoinpti}}}>0$
depending on $K$ with the following property.
Let $n\geq 2$ and let $A$ be an $n\times n$ random matrix as in Lemma~\ref{aiueyfgiwyetfciavvbva}. 
Let $a>0$ be a parameter (possibly depending on $n$).
Then, conditioned on the event
$\Event_{\text{\tiny\ref{aiueyfgiwyetfciavvbva}}}$,
for every choice of non-empty subsets $I,J\subset[n]$ 
we have
$$
\|A_{I,J}\|_{\infty\to 2}\leq C_{\text{\tiny\ref{ajhgbvoinpti}}}
\,\max\Big(\max\limits_{i\in I}\sqrt{\Exp\,\big\|\row_i(A_{I,J})\big\|_2^2},
\max\limits_{j\in J}\sqrt{\Exp\,\big\|\col_j(A_{I,J})\big\|_2^2}\Big)
\,\sqrt{|J|}\,\log^{3/2} n.
$$
\end{lemma}
\begin{proof}
Condition on any realization of $A$ from $\Event_{\text{\tiny\ref{aiueyfgiwyetfciavvbva}}}$,
and let $I,J$ be subsets of $[n]$ as in the statement of the corollary.
Denote
$$
a:=\max\Big(\max\limits_{i\in I}\sqrt{\Exp\,\big\|\row_i(A_{I,J})\big\|_2^2},
\max\limits_{j\in J}\sqrt{\Exp\,\big\|\col_j(A_{I,J})\big\|_2^2}\Big).
$$
Partition the set of indices $I$ according to the magnitudes of expected square 
norms of the rows:
\begin{align*}
&I_k:=\big\{i\in I:\;\Exp\,\big\|\row_i(A_{I,J})\big\|_2^2\in (4^{-k}a^2,4^{-k+1}a^2]\big\},\quad k=1,2,\dots;\\
&I_\infty:=\big\{i\in I:\;\Exp\,\big\|\row_i(A_{I,J})\big\|_2^2=0\big\}.
\end{align*}
In view of the definition of $\Event_{\text{\tiny\ref{aiueyfgiwyetfciavvbva}}}$,
we have for every $k<\infty$ such that $I_k\neq\emptyset$:
$$
\|A_{I_k,J}\|_{\infty\to 2}\leq C_{\text{\tiny\ref{aiueyfgiwyetfciavvbva}}}\,2^{-k+1}a\,\sqrt{\max(|I_k|,|J|)\,\log n}
\leq 2C_{\text{\tiny\ref{aiueyfgiwyetfciavvbva}}}\,a\,\sqrt{|J|\,\log n},
\quad k=1,2,\dots,
$$
where we used that
$$
|I_k|\cdot 4^{-k}a^2\leq \Exp\,\|A_{I_k,J}\|_{HS}^2\leq \Exp\,\|A_{I,J}\|_{HS}^2\leq |J|\,a^2.
$$
On the other hand, for every $k<\infty$ we have
$$
\|A_{I_k,J}\|_{\infty\to 2}\leq \sqrt{n}\,\|A_{I_k,J}\|_{HS}\leq 2^{-k+1}\,a\,n.
$$
Therefore,
$$
\|A_{I,J}\|_{\infty\to 2}\leq 2C_{\text{\tiny\ref{aiueyfgiwyetfciavvbva}}}\,a\,\sum_{k=1}^{\infty}
\min\big(\sqrt{|J|\,\log n},
2^{-k}\,n\big).
$$
The result follows.
\end{proof}

\begin{proof}[Proof of Proposition~\ref{itfufytcuqtrcutqwxi}]
Condition on any realization of $A$ from $\Event_{\text{\tiny\ref{aiueyfgiwyetfciavvbva}}}$,
fix any non-empty
$I,J\subset[n]$, and define
$$
a:=\max\Big(\max\limits_{i\in I}\sqrt{\Exp\,\big\|\row_i(A_{I,J})\big\|_2^2},
\max\limits_{j\in J}\sqrt{\Exp\,\big\|\col_j(A_{I,J})\big\|_2^2}\Big).
$$
Let $x$ be any unit vector in $\C^J$.
We will show that $\|A_{I,J}\,x\|_2=O_K(a\,\log^{5/2} n)$
which will imply that $\|A_{I,J}\|=O_K(a\,\log^{5/2} n)$.
Without loss of generality, $J=\{1,\dots,|J|\}$, and $|x_1|\geq|x_2|\geq\dots\geq|x_{|J|}|$.
Define
$$
J_k:=J\cap\{2^{k-1},\dots,2^k-1\},\quad k=1,2,\dots,
$$
and note that the coordinate projection of $x$ onto $\C^{J_k}$
(which we will denote by  $x^{(k)}$) has the $\|\cdot\|_\infty$--norm
at most $2^{1/2-k/2}$ whereas $|J_k|=2^{k-1}$, $k=1,2,\dots$.
Applying Lemma~\ref{ajhgbvoinpti}, we then get
$$
\|A_{I,J_k}\,x^{(k)}\|_2\leq 2^{1/2-k/2}\|A_{I,J_k}\|_{\infty\to 2}
\leq C_{\text{\tiny\ref{ajhgbvoinpti}}}\,a\,\log^{3/2} n.
$$
The result follows by observing that
$$
\|A_{I,J}\,x\|_2\leq\sum_{k=1}^\infty \|A_{I,J_k}\,x^{(k)}\|_2,
$$
and that $x^{(k)}=0$ whenever $2^{k-1}>n$.
\end{proof}

\section{The smallest singular value of inhomogeneous matrices}\label{asytfvufytviyvbbntr}

Fix a large $n$ and parameters $K\geq 1$, $\rho_0>0$.
Consider a random matrix $A=V\odot W$, 
where $W$ is an $n\times n$ real or complex random matrix,
$V$ is a matrix with non-random non-negative elements, and ``$\odot$'' is the Hadamard (entry-wise)
product.
For the rest of the section, we define
$$
\sigma^*:=
\max\limits_{i,j\leq n}|V_{i,j}|,
$$
and
$$
\sigma:=
\max\bigg(\max\limits_{j\leq n}\sqrt{\sum_{i=1}^n V_{i,j}^2},
\max\limits_{i\leq n}\sqrt{\sum_{j=1}^n V_{i,j}^2}\bigg).
$$

We will assume that the matrix $W$ 
satisfies
\begin{Assumption}[Independence and moments]\label{thkrnhoiwgbouibuve}
The entries of $W$ are independent,
$K$--sub\-gaussian, centered, of unit absolute second moments.
\end{Assumption}
\begin{Assumption}[Density]\label{agedytewrcuaciarebv}\hspace{0cm}
\begin{itemize}
\item Either the entries of $W$ are real with distribution densities
uniformly bounded above by $\rho_0$,
\item Or the entries are complex with independent real and imaginary
parts having distribution
densities
uniformly bounded above by $\rho_0$.
\end{itemize}
\end{Assumption}

\begin{Remark}\label{iytvuyovbnotwrhoubh}
Note that, assuming $n$ is sufficiently large, everywhere on event
$\Event_{\text{\tiny\ref{itfufytcuqtrcutqwxi}}}$ we have
$$
\|A\|\leq C_{\text{\tiny\ref{itfufytcuqtrcutqwxi}}}\,\sqrt{n}\sigma^*\,\log^{5/2} n< n\sigma^*.
$$
\end{Remark}

\bigskip

The next theorem is the main result of the paper.
\begin{theorem}\label{qiuyefqiwgfqvuq}
For every $K,R\geq 1$, $\rho_0>0$ and $\kappa\in(0,1]$
there is $n_{\text{\tiny\ref{qiuyefqiwgfqvuq}}}\in\N$
depending on $K$, $R$, $\rho_0$, and $\kappa$ with the following property.
Let $n\geq n_{\text{\tiny\ref{qiuyefqiwgfqvuq}}}$,
and let the random matrix $W$ satisfy Assumptions~\ref{thkrnhoiwgbouibuve} and~\ref{agedytewrcuaciarebv}.
Assume additionally that $z$ is a non-zero complex number such that
$$
|z|\geq \max\Big(\sigma^*\,n^{2\kappa},\frac{\sigma}{R}\Big).
$$
Then with probability at least $1-2n^{-2}$ the smallest singular value of the matrix $A-z\,\Id=V\odot W-z\,\Id$
satisfies
$$
s_{\min}(A-z\,\Id)\geq |z|\,\exp\bigg(-\frac{n^{1+3\kappa}(\sigma^*)^2}{|z|^2}\bigg)
\geq |z|\,\exp\bigg(-R^2\,n^{3\kappa}\,\Big(
\frac{\sqrt{n}\sigma^*}{\sigma}\Big)^2\bigg).
$$
\end{theorem}
\begin{Remark}
The probability bound $1-2n^{-2}$ can be replaced with $1-n^{-C}$ for arbitrary
universal constant $C>0$ without any changes in proof.
\end{Remark}
\begin{Remark}
We leave it as an open problem to verify the conclusion of Theorem~\ref{qiuyefqiwgfqvuq}
without Assumption~\ref{agedytewrcuaciarebv} on $W$.
\end{Remark}

We expect that the estimate on $s_{\min}(A-z\,\Id)$ in Theorem~\ref{qiuyefqiwgfqvuq} 
can be significantly improved if the variance profile $V$ is doubly stochastic;
see \cite[Section~2.6]{CHNR1} for related open problems.
On the other hand, without any extra assumptions on $A$ and $z$
the bound is close to optimal as the next example shows.

\begin{Example}
Let $d\in\N$ be a parameter such that $n/d$ is an integer, and let $V$ be a block matrix of the form
$$
V=\begin{pmatrix}
0_{d\times d} & V_{1} & 0_{d\times d} & \dots & 0_{d\times d} \\
0_{d\times d} & 0_{d\times d} & V_{2} & \dots & 0_{d\times d} \\
\dots & \dots & \dots & \dots \\
0_{d\times d} & 0_{d\times d} & 0_{d\times d} & \dots & 0_{d\times d}
\end{pmatrix},
$$
where each block $V_{\ell}$, $1\leq \ell<n/d$, is a $d\times d$ matrix of ones.
Set $z:=\sqrt{d}/4$, let $G$ be the $n\times n$ standard real Gaussian matrix,
and denote by $G_1,G_2,\dots,G_{n/d-1}$ be the $d\times d$ submatrices of $G$
corresponding to $V_\ell$'s.

Fix any unit vector $x_{n/d}\in\R^d$. We define $d$--dimensional vectors $x_1,\dots,x_{n/d-1}$
recursively via relations
$$
-z\,x_\ell=G_\ell\,x_{\ell+1},\quad \ell=1,\dots,n/d-1.
$$
Conditioned on any realization of $x_{\ell+1},\dots,x_{n/d}$,
the random vector $G_\ell\,x_{\ell+1}$ has Euclidean norm at least $\sqrt{d}\,\|x_{\ell+1}\|_2/2$
with probability at least $1-2\exp(-cd)$ (for a universal constant $c>0$).
Applying this estimate recursively and using the definition of $z$, we get
$$
\|x_1\|_2\geq 2^{n/d-1}
$$
with probability at least $1-2d\exp(-cd)$.
Assuming the last estimate, we have
$$
s_{\min}(V\odot G-z\,\Id)\leq 2^{1-n/d}\big\|(V\odot G-z\,\Id)\,
(\oplus_{i\leq n/d}\,x_i)\big\|_2=2^{1-n/d}\,|z| \leq\sqrt{d}\,2^{1-n/(16|z|^2)}.
$$
Thus, under the assumption $d=\omega(1)$ we get
$$
s_{\min}(V\odot G-z\,\Id)= \exp\big(-\Omega(n/|z|^2)\big)
$$
with probability $1-o(1)$.
\end{Example}

\subsection{Normal vectors}

In this subsection, we assume that $z$ is a non-zero complex number.

\begin{defi}
For every subset $J\subset[n]$ of size at least $2$ and every $j\in J$ denote by
$\normal_{J,j}$ a random unit vector in $\C^J$ orthogonal to the linear span
of $\col_k(A_J-z\,\Id)$, $k\in J\setminus\{j\}$ and measurable w.r.t the sigma-field
generated by those columns ($\normal_{J,j}$
is not uniquely defined, but we fix some version
of the vector for the rest of the proof). 
\end{defi}

\medskip

\begin{lemma}\label{avfiahcgiaygvciywafiyg}
For every $\rho_0>0$ and $K\geq 1$ and assuming $n$ is sufficiently large,
the following holds. Let
\begin{align*}
\Event_{\text{\tiny\ref{avfiahcgiaygvciywafiyg}}}:=\bigg\{
&\big|\langle \normal_{J,j},\col_j(A_J) \rangle\big|\leq
n\,\sqrt{\sum_{k\in J}|(\normal_{J,j})_k|^2\,|(\col_j(V_J))_k|^2}\\
&\mbox{for all $J\subset[n]$ with $|J|\geq 2$, and all $j\in J$}
\bigg\}
\end{align*}
Then $\Prob(\Event_{\text{\tiny\ref{avfiahcgiaygvciywafiyg}}})
\geq 1-\exp(-c_{\text{\tiny\ref{avfiahcgiaygvciywafiyg}}}\,n^2)$,
where $c_{\text{\tiny\ref{avfiahcgiaygvciywafiyg}}}>0$ depends only on $K$.
\end{lemma}
\begin{proof}
Fix for a moment any $J\subset[n]$ with $|J|\geq 2$, and any $j\in J$,
and condition on any realization of $\col_k(A_J-z\,\Id)$, $k\in J\setminus\{j\}$
(so that $\normal_{J,j}$ is fixed but $\col_j(A_J)$ is random).
We will assume that $(\normal_{J,j})_k\,(\col_j(V_J))_k\neq 0$ for some $k\in J$.
The product $\langle \normal_{J,j},\col_j(A_J)\rangle$
can be viewed as a linear combination of mutually independent $K$--subgaussian centered variables
having unit absolute second moments, with coefficients
$$
(\col_j(V_J))_k(\normal_{J,j})_k,\quad k\in J.
$$
Hence, $\langle \normal_{J,j},\col_j(A_J)\rangle$ is centered, with absolute second moment
$$
\sum_{k\in J}|(\normal_{J,j})_k|^2\,|(\col_j(V_J))_k|^2,
$$
and the normalized variable
$$
\frac{\langle \normal_{J,j},\col_j(A_J)\rangle}{\sqrt{\sum_{k\in J}|(\normal_{J,j})_k|^2\,|(\col_j(V_J))_k|^2}}
$$
is $CK$--subgaussian for some universal constant $C>0$
(see Theorem~\ref{akjfgifyutviywvgcywv}).
The definition of a subgaussian variable then implies
$$
\Prob\bigg\{
\big|\langle \normal_{J,j},\col_j(A_J) \rangle\big|>
n\,\sqrt{\sum_{k\in J}|(\normal_{J,j})_k|^2\,|(\col_j(V_J))_k|^2}\bigg\}
\leq 2\exp\big(-cn^2/K^2\big),
$$
for a universal constant $c>0$.
Assuming that $n$ is sufficiently large and taking the unit bound over all possible choices of $J$ and $j$,
we get the result.
\end{proof}

\begin{Remark}
The ``good'' event defined in the last lemma is designed to rule out the situation
when both $\langle \normal_{J,j},\col_j(A_J-z\,\Id)\rangle$ and
$$
\sqrt{\sum_{k\in J}|(\normal_{J,j})_k|^2\,|(\col_j(V_J))_k|^2}
$$
are ``small'' whereas the $j$--th component of the normal
$(\normal_{J,j})_j$ is ``large''; see the proof of Lemma~\ref{akufvawiyfvufcuytvu} below
(specifically, formula~\eqref{ahgvuaytvutcwcvut}) for details.
\end{Remark}

\subsection{A directed graph on submatrices}\label{alrberufvifyviyqviv}

Suppose that the matrices $A,V,W$
satisfy Assumptions~\ref{thkrnhoiwgbouibuve} and~\ref{agedytewrcuaciarebv},
and assume additionally that $\sigma^*\leq |z|$.
Define parameters
$$
\delta:=\frac{|z|}{n},\quad L:=n^{-\kappa}\,|z| ,\quad \beta:=n^{-3}.
$$
Further, let $\tilde V$ be the $n\times n$ matrix obtained from $V$ by replacing elements
of $V$ of magnitude at most $\delta$ with zeros, namely,
$$
\tilde V_{i,j}:=
\begin{cases}
V_{i,j},&\mbox{if $V_{i,j}>\delta$};\\
0,&\mbox{otherwise.}
\end{cases}
$$

In this subsection, we construct a directed graph $\mathcal G$,
in which each vertex is a square random matrix of the form $A_J-z\,\Id$ for some $J\subset[n]$.
In particular, the constructed graph would satisfy the following properties:
\begin{itemize}
\item Any out-neighbor of any vertex 
is a proper principal submatrix of that vertex;
\item All vertices of $\mathcal G$ but one have at least one in-neighbor;
\item The unique vertex that has only out-neighbors and no in-neighbors (the ``source'') is the matrix
$A-z\,\Id$.
\end{itemize}
Note that the above properties imply that
$\mathcal G$ has no {\it directed} cycles, and is connected (although not strongly connected).
The edges of the graph $\mathcal G$ are labeled (see further).
The vertices of the graph which have only in-neighbors but no out-neighbors
will be called {\it terminals}.

The graph $\mathcal G$ is constructed from the source $A-z\,\Id$ to terminals,
according to the following process.
Assume that for some $J\subset[n]$, the matrix $A_J-z\,\Id$
has been added to the vertex set of $\mathcal G$.
Denote by $J'\subset J$ the subset of all indices $j\in J$ such that
$$
\sum_{i\in J}(V_J)_{i,j}^2\leq L^2,
$$
and let $J''$ be the complement of $J'$ in $J$.
One can interpret $J'$ and $J''$ as sets of indices corresponding to ``sparse'' 
and ``dense'' columns
of $A_J-z\,\Id$, respectively.
For each $j\in J''$, we let $U_j$ be the set of all indices $i\in J$
with $\tilde V_{i,j}\neq 0$ (note that in view of the definition of $\delta$,
$U_j\neq\emptyset$ for every $j\in J''$). Observe that for every $j\in J''$,
$$
L^2<\sum_{i\in J}(V_J)_{i,j}^2
\leq \sum_{i\in J}(\tilde V_J)_{i,j}^2+|J|\,\delta^2
\leq |U_j|\,(\sigma^*)^2+|J|\,\delta^2,
$$
and hence
\begin{equation}\label{aouyevfiuytfviygcviqyg}
|U_j|> \frac{L^2}{2(\sigma^*)^2}.
\end{equation}
\begin{itemize}
\item[(a)] If $J''$ is empty then we declare $A_J-z\,\Id$ to be a terminal of $\mathcal G$.
\item[(b)] Otherwise,

\noindent For each $j\in J''$,
\begin{itemize}
\item If $A_{J\setminus (U_j\cup\{j\})}-z\,\Id$ is not yet in the vertex set of $\mathcal G$
then we add the matrix to the vertex set (here, in case $U_j\cup\{j\}=J$ the
matrix $A_{J\setminus (U_j\cup\{j\})}-z\,\Id$ is empty).
\item We draw a directed edge from $A_J-z\,\Id$ to $A_{J\setminus (U_j\cup\{j\})}-z\,\Id$
and label the edge by the column index $j$.
\end{itemize}
Furthermore, if $J'$ is non-empty then
\begin{itemize}
\item If $A_{J'}-z\,\Id$ is not yet in the vertex set of $\mathcal G$
then we add the matrix to the vertex set.
\item If there is no edge from $A_J-z\,\Id$ to $A_{J'}-z\,\Id$ yet then we draw the directed edge and 
assign empty label ``'' to it.
\end{itemize}
Repeat (a)--(b) for each of the newly added vertices of $\mathcal G$.
\end{itemize}
Note that although the vertices of $\mathcal G$ are random matrices, the edge structure of the
graph is deterministic i.e does not depend on a realization of $A$.

The terminals of the graph are of two types: either the empty matrix $[\;]$
or a submatrix $A_J-z\,\Id$ such that 
$$
\sum_{i\in J}(V_J)_{i,j}^2\leq L^2
$$
for all $j\in J$.
We will refer to the latter as {\bf non-empty terminals}.
Note further that for every vertex $A_J-z\,\Id$ which is not a terminal, we have
$|J|\geq L^2/(\sigma^*)^2$.
We refer to Figures~\ref{aknboewiybwiubvieuyb} and~\ref{auhyuaytewutcwrv}
for examples of the graph $\mathcal G$.

\begin{Remark}\label{aihbvieruvhbobitrnboirn}
In view of
\eqref{aouyevfiuytfviygcviqyg}, for every out-neighbor $v'$ of $v$ which is not a terminal,
the difference of linear dimensions of the matrices $v$ and $v'$ is greater than
$\frac{L^2}{2(\sigma^*)^2}$.
Therefore, any directed path connecting the source with a terminal
has length at most $\lceil 2n(\sigma^*)^2/L^2\rceil$.
\end{Remark}

\begin{figure}[h]
\caption{Example of a graph $\mathcal G$ constructed for a $4\times 4$
matrix $A-z\,\Id$, where $A$ is lower triangular, with non-zero entries having unit variances,
and with parameter $L^2:=2$. The non-zero entries are represented by ``x''.
With the notation from the construction process for $\mathcal G$,
for the source $A-z\,\Id$ we have $J'=\{3,4\}$, $J''=\{1,2\}$, and $U_1=\{1,2,3,4\}$, $U_2=\{2,3,4\}$.
The source has three out-neighbors: the empty matrix and two non-empty terminals
$A_{J'}-z\,\Id$ and $A_{\{1\}}-z\,\Id$ (we represent the submatrices by ``crossing out''
the complementary rows and columns).}\label{aknboewiybwiubvieuyb}
\centering
\includegraphics[width=0.5\textwidth]{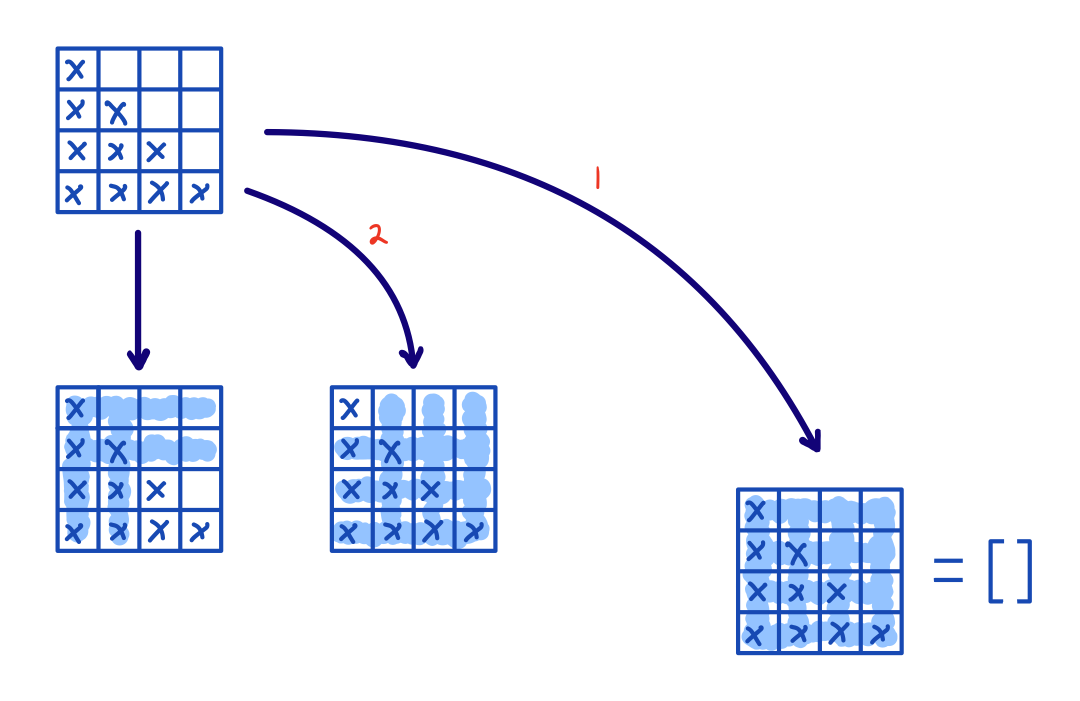}
\end{figure}

\begin{figure}[h]
\caption{Graph $\mathcal G$ for a $6\times 6$
matrix $A-z\,\Id$, where $A$ is a periodic band matrix
with non-zero entries having unit variances,
and parameter $L^2:=2$. 
We show only two terminals (out of seven) in the picture.}
\label{auhyuaytewutcwrv}
\centering
\includegraphics[width=0.8\textwidth]{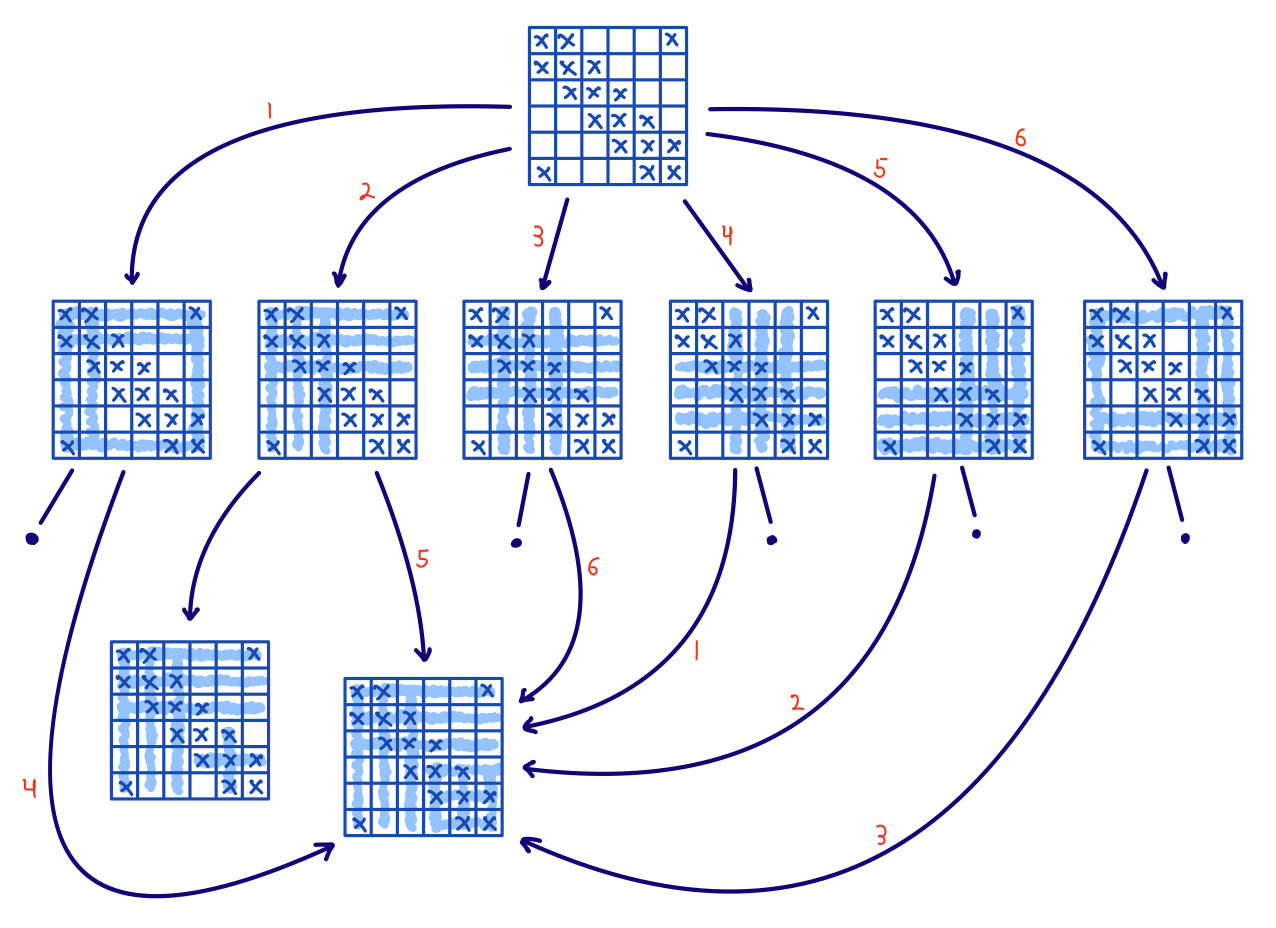}
\end{figure}

The basis of the proof of Theorem~\ref{qiuyefqiwgfqvuq}
is a ``deconstruction'' argument in which the event that $s_{\min}(A-z\,\Id)$ small
is associated with a random path on $\mathcal G$ leading from the source to a terminal,
and each submatrix along the path satisfies certain bound on its smallest singular value.
Those estimates, in turn, are related to distribution of distances between a column of a submatrix
to the span of other columns. Ultimately, the event ``$s_{\min}(A-z\,\Id)$ is small''
is estimated by a product of probabilities that distances between certain
random vectors and random subspaces are small, plus the event that some of the
non-empty terminals are ill-conditioned.
Lemma~\ref{akufvawiyfvufcuytvu} which lies at the core of the argument, establishes 
a single step of the deconstruction process. Then, in Proposition~\ref{aiytevutvcutvirvuveo}, the multiple steps
are aggregated with help of auxiliary data structures specifying a way to descend through $\mathcal G$.

\begin{notation}
In what follows, we adopt the convention $s_{\min}([\;]):=+\infty$.
\end{notation}

\begin{notation}
We define numbers
$$
t_p:=2^{-p},\quad p=0,1,2,\dots;\quad t_{-1}:=+\infty.
$$
Further, let $p_0$ be the largest number in $\{0,\dots,n\}$ such that $t_{p_0}\geq \beta\,n^{2}$,
and let $\tilde p_0$ be the smallest number in $\{0,\dots,n\}$ such that $t_{\tilde p_0}
\leq (32\,n^{3}/\beta)^{-1}$. 
\end{notation}

\begin{lemma}\label{akufvawiyfvufcuytvu}
Condition on any realization of the matrix $A$ from
$\Event_{\text{\tiny\ref{avfiahcgiaygvciywafiyg}}}
\cap\Event_{\text{\tiny\ref{itfufytcuqtrcutqwxi}}}$. 
Let $N= A_J-z\,\Id$ be a vertex of
$\mathcal G$ which is not a terminal,
such that $s_{\min}(N)\leq t_q\,|z|$ for some $q\in\{0,\dots,n\}$.
Then at least one of the following holds:
\begin{itemize}
\item There is $r\in\{-1,\dots,n\}$ and an out-edge $e$ for $N$ with a non-empty label $j$
such that the corresponding out-neighbor $A_{\tilde J}-z\,\Id$ of $N$ satisfies
$s_{\min}(A_{\tilde J}-z\,\Id)\leq t_{r}\,|z|$, and
$$
\bigg|\frac{\langle\normal_{J,j},\col_j(A_J-z\,\Id)\rangle}
{\sqrt{\sum_{k\in J}|(\normal_{J,j})_k|^2\,|(\col_j(V_J))_k|^2}}\bigg|\leq t_{\max(q-r-\tilde p_0,-1)}.
$$
\item There is a zero-labeled
edge connecting $A_J-z\,\Id$
to a non-empty terminal $A_{\tilde J}-z\,\Id$
such that
$$
s_{\min}(A_{\tilde J}-z\,\Id)< t_{p_0}\,|z|.  
$$
\end{itemize}
\end{lemma}
\begin{proof}
Without loss of generality, we can assume that $q\geq \tilde p_0-1$,
as otherwise the first assertion of the lemma holds with $r=-1$.
In particular, this implies $\frac{16\,n^3\,t_q}{\beta}\leq 1$.

Let subsets $J',J''\subset J$ and sets $U_j$, $j\in J''$ be as in the definition of $\mathcal G$.
The assumption on $s_{\min}(N)$ implies that there is a unit vector $x\in\C^J$ such that
$$
\|(A_J-z\,\Id)x\|_2\leq t_q\,|z|.
$$

\medskip

{\bf(I)} If there is $j\in J''$ such that $|x_j|\geq \beta$ then the distance from $\col_j(A_J-z\,\Id)$
to the linear span of columns $\col_k(A_J-z\,\Id)$, $k\in J\setminus\{j\}$,
is at most $t_q\,|z|/\beta$, i.e
\begin{equation}\label{ajhbfufviweyfviyfviyv}
\big|
-z\,(\normal_{J,j})_j+\langle \normal_{J,j},\col_j(A_J)\rangle
\big|
=
\big|\langle \normal_{J,j},\col_j(A_J-z\,\Id)\rangle \big|\leq t_q\,|z|/\beta.
\end{equation}
In view of our conditioning on $\Event_{\text{\tiny\ref{avfiahcgiaygvciywafiyg}}}$, this implies
\begin{equation}\label{ahgvuaytvutcwcvut}
|z|\,\big|(\normal_{J,j})_j\big|\leq t_q\,|z|/\beta+n\,\sqrt{\sum_{k\in J}|(\normal_{J,j})_k|^2\, |(\col_j(V_J))_k|^2}.
\end{equation}
Let $\tilde r$ be the largest number in $\{0,1,\dots,n\}$ such that
$$
\sqrt{\sum_{k\in J}|(\normal_{J,j})_k|^2\, |(\col_j(V_J))_k|^2}\leq t_{\tilde r}\,|z|
$$
(the assumption that $\sigma^*\leq |z|$ guarantees that such number $\tilde r$ exists).
Then, by \eqref{ajhbfufviweyfviyfviyv} and
\eqref{ahgvuaytvutcwcvut}, 
\begin{equation}\label{kajhgvfaytfvutrytewrdfad}
\big|(\normal_{J,j})_j\big|\leq t_q/\beta+n\,\tilde t_r
\end{equation}
and
\begin{equation}\label{akhgvfjwfvcwwqfutcbfb}
|\langle\normal_{J,j},\col_j(A_J-z\,\Id)\rangle|
\leq\max\bigg(t_n\,|z|, \sqrt{\sum_{k\in J}|(\normal_{J,j})_k|^2\, |(\col_j(V_J))_k|^2}\bigg)
\cdot \frac{2\,t_q}{\beta\,t_{\tilde r}}.
\end{equation}
Observe that the condition
$$
\sum_{k\in U_j}|(\normal_{J,j})_k|^2 \delta^2\leq
\sum_{k\in J}|(\normal_{J,j})_k|^2\, |(\col_j(V_J))_k|^2\leq t_{\tilde r}^2\,|z|^2
$$
implies that the coordinate projection $v$ of the vector $\normal_{J,j}$ onto $\C^{U_j}$
has the Euclidean norm at most $t_{\tilde r}\,|z|/\delta=n\,t_{\tilde r}$,
and, in view of \eqref{kajhgvfaytfvutrytewrdfad}, the 
coordinate projection $\tilde v$ of $\normal_{J,j}$ onto $\C^{J\setminus (U_j\cup\{j\})}$ satisfies
\begin{equation}\label{ayvauyuwtcwutwcutafu}
\|\tilde v\|_2\geq 1-\|v\|_2-\big|(\normal_{J,j})_j\big|
\geq 1-2n\,t_{\tilde r}-t_q/\beta,
\end{equation}
and, recalling conditioning on $\Event_{\text{\tiny\ref{itfufytcuqtrcutqwxi}}}$ and Remark~\ref{iytvuyovbnotwrhoubh},
\begin{equation}\label{auhbfieaubwifubiqwyg}
\begin{split}
\big\|
\tilde v (A-z\,\Id)_{J\setminus (U_j\cup\{j\}),J\setminus\{j\}}
\big\|_2&\leq \|v\|_2\,\|A-z\,\Id\|+\big|(\normal_{J,j})_j\big|\,\|A-z\,\Id\|\\
&\leq \big(n\,\sigma^*+|z|\big)(t_q/\beta+2n\,t_{\tilde r})\\&\leq (n+1)\,|z|\,(2n t_{\tilde r}+t_q/\beta).
\end{split}
\end{equation}
At this point, we consider several subcases.

If $8n^3\,t_{\tilde r}\geq 1$ then
\eqref{akhgvfjwfvcwwqfutcbfb} implies
$$
\frac{|\langle\normal_{J,j},\col_j(A_J-z\,\Id)\rangle|}{\sqrt{\sum_{k\in J}|(\normal_{J,j})_k|^2\,|(\col_j(V_J))_k|^2}}\leq
 \frac{16\,n^3\,t_q}{\beta},
$$
and the first assertion of the lemma holds with $r:=-1$.

Next, assume that $8n^3\,t_{\tilde r}< 1$.
Note that the assumption and our condition on $q$ imply, in particular, that $2n t_{\tilde r}+t_q/\beta\leq 1/2$,
and from \eqref{ayvauyuwtcwutwcutafu}
and \eqref{auhbfieaubwifubiqwyg} we have
\begin{align*}
s_{\min}(A_{J\setminus (U_j\cup\{j\})}-z\,\Id)
&\leq s_{\min}\Big(\big((A-z\,\Id)_{J\setminus (U_j\cup\{j\}),J\setminus\{j\}}\big)^\top\Big)\\
&\leq \frac{\big\|
\tilde v (A-z\,\Id)_{J\setminus (U_j\cup\{j\}),J\setminus\{j\}}
\big\|_2}{\|\tilde v\|_2}\\
&\leq (n+1)\,|z|\,(4n t_{\tilde r}+2t_q/\beta).
\end{align*}
If $4n t_{\tilde r}\leq 2t_q/\beta$ then the last estimate implies
$s_{\min}(A_{J\setminus (U_j\cup\{j\})}-z\,\Id)\leq
4(n+1)\,|z|\, t_q/\beta$, and the first assertion of the lemma holds with $r:=q-\tilde p_0+1$.
On the other hand, if $4n t_{\tilde r}> 2t_q/\beta$ (and, in particular, $t_{\tilde r}\geq 2t_n$)
then
$$
s_{\min}(A_{J\setminus (U_j\cup\{j\})}-z\,\Id)\leq 8n(n+1)\,|z|\,\,t_{\tilde r},
$$
and, by \eqref{akhgvfjwfvcwwqfutcbfb},
$$
\frac{|\langle\normal_{J,j},\col_j(A_J-z\,\Id)\rangle|}
{\sqrt{\sum_{k\in J}|(\normal_{J,j})_k|^2|(\col_j(V_J))_k|^2}}
\leq \frac{2\,t_q}{\beta\,t_{\tilde r}}.
$$
Thus, the first assertion holds with $r:=\tilde r-\lceil\log_2 (8n(n+1))\rceil\geq 0$.

\medskip

{\bf(II)} If $|x_j|< \beta$ for all $j\in J''$ (and also if $J''$ is empty), we have $J'\neq\emptyset$ and
$$
\|(A-z\,\Id)_{J,J'}\,x'\|_2\leq t_q\,|z|+\|(A-z\,\Id)_{J,J''}\,x''\|_2\leq t_q\,|z|+\beta\sqrt{n}\,\|A-z\,\Id\|,
$$
where $x'$ and $x''$ are the coordinate projections of $x$ onto $\C^{J'}$ and $\C^{J''}$, respectively,
and $\|x'\|\geq 1-\beta\sqrt{n}>1/2$. Note further that since $J'\subset J$, we have
relations
$$
s_{\min}\big((A-z\,\Id)_{J'}\big)\leq s_{\min}\big((A-z\,\Id)_{J,J'}\big)
\leq \frac{\|(A-z\,\Id)_{J,J'}\,x'\|_2}{\|x'\|_2}.
$$
Thus, in view of conditioning on $\Event_{\text{\tiny\ref{itfufytcuqtrcutqwxi}}}$, the smallest singular value of $(A-z\,\Id)_{J'}$ is at most
$$2t_q\,|z|+2\beta\sqrt{n}\,\|A-z\,\Id\|\leq
2t_q\,|z|+2\beta\sqrt{n}\,(n+1)\,|z|
< \beta n^{2}\,|z|.$$
The second assertion of the lemma follows.
\end{proof}

\begin{Remark}
Lemma~\ref{akufvawiyfvufcuytvu} is one of two places in the proof
(together with the uniform estimate on the smallest singular values of non-empty terminals
in Proposition~\ref{alejhfbwfuvufceutqwfvutfyv})
where a lower bound on $|z|$ is crucial. In formula~\eqref{kajhgvfaytfvutrytewrdfad},
the lower bound is used to control from above the absolute value of the $j$--th
component of $\normal_{J,j}$. As a consequence of that restriction, and
in qualitative terms,
the situation where simultaneously
(a) $|\langle \normal_{J,j},\col_j(A_J-z\,\Id)\rangle|$ is small, (b) the normalized
inner product $\big|\frac{\langle\normal_{J,j},\col_j(A_J-z\,\Id)\rangle}
{\sqrt{\sum_{k\in J}|(\normal_{J,j})_k|^2\,|(\col_j(V_J))_k|^2}}\big|$ is large,
and (c) $s_{\min}(A_{J\setminus (U_j\cup\{j\})}-z\,\Id)$ is large, is impossible
when conditioned on the event $\Event_{\text{\tiny\ref{avfiahcgiaygvciywafiyg}}}$.
\end{Remark}

\bigskip

\begin{defi}[Data structures]
Denote by $\mathcal M$ the collection of all data structures of the form
$\mathcal S=(\mathcal P=(\mathcal P[\ell])_{\ell=0}^d,(r_\ell)_{\ell=0}^{d})$
(where $d$ is not fixed and depends on $\mathcal S$) such that
\begin{itemize}
\item $\mathcal P$ is a valid path on $\mathcal G$ starting at the source $\mathcal P[0]
=A-z\,\Id$
and ending at a terminal of the graph; 
\item $r_0,\dots,r_d$ are numbers in $\{-1,0,\dots,n\}$.
\end{itemize}
\end{defi}

The next lemma is an immediate consequence of Remark~\ref{aihbvieruvhbobitrnboirn}:
\begin{lemma}[Size of $\mathcal M$]\label{akgvutrcytqrcwxtqwevuyc}
The total number of data structures in $\mathcal M$ is bounded above by
$$
\big(n(n+2)\big)^{\lceil 2n(\sigma^*)^2/L^2\rceil}.
$$
\end{lemma}

\begin{defi}[$A$-compatible data structures]
Let $\mathcal S=((\mathcal P[\ell])_{\ell=0}^d,(r_\ell)_{\ell=0}^{d})\in\mathcal M$.
We say that $\mathcal S$ is $A$-compatible for a given realization of $A$ if
$s_{\min}(\mathcal P[\ell])\leq r_\ell\,|z|$ for all $0\leq \ell\leq d$.
\end{defi}

\begin{prop}\label{aiytevutvcutvirvuveo}
Let $r_0\leq n$ be the largest integer such that
$\exp\big(-\frac{n^{1+3\kappa}(\sigma^*)^2}{|z|^2}\big)\leq t_{r_0}$.
Condition on any realization of the random matrix $A$ from
$\Event_{\text{\tiny\ref{avfiahcgiaygvciywafiyg}}}
\cap\Event_{\text{\tiny\ref{itfufytcuqtrcutqwxi}}}$
such that $s_{\min}(A-z\,\Id)\leq t_{r_0}\,|z|$. Then there exists an $A$-compatible structure
$\mathcal S=(\mathcal P=(\mathcal P[\ell])_{\ell=0}^d,(r_\ell)_{\ell=0}^{d})$
having the following property.
Whenever $1\leq\ell\leq d$ is such that $r_{\ell}< r_{\ell-1}-\tilde p_0$ then necessarily
\begin{itemize}
\item Either 
$$
\bigg|\frac{\langle\normal_{J_{\ell-1},j_{\ell-1}},\col_{j_{\ell-1}}(A_{J_{\ell-1}}-z\,\Id)\rangle}
{\sqrt{\sum_{k\in J_{\ell-1}}|(\normal_{J_{\ell-1},j_{\ell-1}})_k|^2
|(\col_{j_{\ell-1}}(V_{J_{\ell-1}}))_k|^2}}\bigg|\leq t_{r_{\ell-1}-r_\ell-\tilde p_0},
$$
where $\mathcal P[\ell-1]=A_{J_{\ell-1}}-z\,\Id$, $\mathcal P[\ell]=A_{J_\ell}-z\,\Id$,
and where $j_{\ell-1}$
is the (non-empty) label of the edge connecting $\mathcal P[\ell-1]$
with $\mathcal P[\ell]$,
\item Or, if the first condition does not hold then
$\mathcal P[\ell]=A_{J_\ell}-z\,\Id$ is a non-empty terminal
(and $\ell=d$), and $r_\ell\geq p_0$.
\end{itemize}
\end{prop}
\begin{proof}
We will construct the path $\mathcal P$
and the sequence $(r_\ell)$ in steps, by iteratively applying Lemma~\ref{akufvawiyfvufcuytvu}.
Let $\mathcal P[0]:=A-z\,\Id$ (the number $r_0$ is already defined in the statement of the lemma).
Also denote $J_0:=[n]$ so that $A_{J_0}=A$.
At the first step, by Lemma~\ref{akufvawiyfvufcuytvu}, one of the following must be true:
\begin{itemize}
\item There is $r_1\in\{-1,\dots,n\}$ and an out-neighbor
$\mathcal P[1]:=A_{J_1}-z\,\Id$ of the source 
such that $s_{\min}(A_{J_1}-z\,\Id)\leq t_{r_1}\,|z|$, and
$$
\bigg|\frac{\langle\normal_{J_0,j_0},\col_{j_0}(A_{J_0}-z\,\Id)\rangle}
{\sqrt{\sum_{k\in I_0}|(\normal_{J_0,j_0})_k|^2|(\col_{j_0}(V_{J_0}))_k|^2}}\bigg|
\leq t_{\max(r_0-r_1-\tilde p_0,-1)},
$$
where $j_0$ is the label of the edge connecting the source with $A_{J_1}-z\,\Id$.
If $\mathcal P[1]$ is a terminal 
then we stop the construction.
\item Otherwise, if the first condition does not hold then
there is a non-empty terminal $\mathcal P[1]:=A_{J_1}-z\,\Id$, which is an out-neighbor of the source,
such that
$$
s_{\min}(A_{J_1}-z\,\Id)\leq 
t_{p_0}\,|z|.
$$
We then set $r_1:=p_0$ and stop the construction.
\end{itemize}

\medskip

At $\ell$--th step, $\ell>1$, we are given partially constructed sequences $(\mathcal P[m])_{m=0}^{\ell-1}$
and $(r_m)_{m=0}^{\ell-1}$, where $\mathcal P[\ell-1]$ is not a terminal.
Similarly to the first step,
an application of Lemma~\ref{akufvawiyfvufcuytvu} produces a vertex $\mathcal P[\ell]=A_{J_\ell}-z\,\Id$
which is an out-neighbor of $\mathcal P[\ell-1]$, and a number $r_\ell\in\{-1,0,\dots,n\}$
such that $s_{\min}(\mathcal P[\ell])\leq t_{r_\ell}\,|z|$ and either (a)
$$
\bigg|\frac{\langle\normal_{J_{\ell-1},j_{\ell-1}},\col_{j_{\ell-1}}(A_{J_{\ell-1}}-z\,\Id)\rangle}
{\sqrt{\sum_{k\in J_{\ell-1}}|(\normal_{J_{\ell-1},j_{\ell-1}})_k|^2\,
|(\col_{j_{\ell-1}}(V_{J_{\ell-1}}))_k|^2}}\bigg|\leq t_{\max(r_{\ell-1}-r_\ell-\tilde p_0,-1)},
$$
where $j_{\ell-1}$ is the (non-empty)
label of the edge connecting $\mathcal P[\ell-1]=A_{J_{\ell-1}}-z\,\Id$
with $\mathcal P[\ell]$, or (b) $\ell=d$ and $\mathcal P[\ell]$ is a non-empty
terminal
with $r_d=p_0$. 
The result follows.
\end{proof}

\medskip

The proof of Theorem~\ref{qiuyefqiwgfqvuq}
is accomplished by estimating the probability
that a fixed data structure $\mathcal S=(\mathcal P,(r_\ell)_{\ell=0}^{d})$
from $\mathcal M$ with $r_0:=\big\lfloor
\log_2\,\exp\big(\frac{n^{1+3\kappa}(\sigma^*)^2}{|z|^2}\big)
\big\rfloor$
is $A$--compatible and satisfies conditions stated in Proposition~\ref{aiytevutvcutvirvuveo},
and then taking the union bound over all $\mathcal S\in\mathcal M$ with
the aforementioned choice of $r_0$.
These probability estimates are obtained
as a combination of three ingredients:
\begin{itemize}
\item A probability bound on the event that any of the non-empty terminals $A_J-z\,\Id$ of
$\mathcal G$ are ill-conditioned, specifically, satisfy $s_{\min}(A_J-z\,\Id)\leq t_{p_0}\,|z|$;
\item Standard anti-concentration estimates for linear combinations of independent
variables (see Lemma~\ref{auevfiwuyfviygvcikygviygv});
\item A ``telescopic'' conditioning argument which takes care of the probabilistic dependencies
between the variables
$$\langle \normal_{J_{\ell},j_\ell},
\col_{j_\ell}(A_{J_{\ell}}-z\,\Id)\rangle,\quad 0\leq \ell\leq d.$$
\end{itemize}
In the following subsection, we apply the Gershgorin--type estimates obtained earlier
to address the first item.

\subsection{Invertibility of non-empty terminals}\label{aefiwtvcytarcytefufwa}

\begin{lemma}\label{auyfiytwvfuyvwfuywvycv}
Let $n\geq m\geq 1$, $\kappa\in(0,1]$, $z\in\C$ and let
$B$ be an $m\times m$ random matrix such that
$$
\Exp\,\|\col_j(B)\|_2^2\leq n^{-2\kappa}\,|z|^2,\quad j\leq m.
$$
Then there is a deterministic permutation $\pi: [m]\to [m]$ and a block representation
of the matrix $X:=\big(B_{\pi(i),\pi(j)}\big)_{i,j\in [m]}$,
$$
X
=\begin{pmatrix}
X_{1,1} & X_{1,2} & \dots & X_{1,k} \\
X_{2,1} & X_{2,2} & \dots & X_{2,k} \\
\dots & \dots & \dots & \dots \\
X_{k,1} & X_{k,2} & \dots & X_{k,k}
\end{pmatrix},
$$
satisfying the following conditions:
\begin{itemize}
\item $k\leq \frac{1}{\kappa}+1$.
\item For every $1\leq j\leq i\leq k$, all rows of the matrix $X_{i,j}$
have expected squared Euclidean norms at most $n^{-\kappa}\,|z|^2$.
\end{itemize}
\end{lemma}
\begin{proof}
We will construct the permutation and the block decomposition through an iterative process.
Let $I_0:=[m]$. At first step, we observe that
since the second moment of the Euclidean norm of each column of $B$
is at most $n^{-2\kappa}\,|z|^2$, we have
$
\Exp\,\|B\|_{HS}^2\leq m\,n^{-2\kappa}\,|z|^2.
$
Consequently, the number of rows $\row_i(B)$ with 
$$
\Exp\,\|\row_i(B)\|_2^2\geq n^{-\kappa}\,|z|^2
$$
is at most $m\,n^{-\kappa}$. Let $I_1$ be the set of all indices $i\in [m]$ such that the last inequality holds.

At $\ell$--th step ($\ell>1$), we are given a finite nested sequence of subsets
$[m]=I_0\supset I_1\supset \dots\supset I_{\ell-1}$ where $I_{\ell-1}$ is non-empty.
Denote by $B'$ the $I_{\ell-1}\times I_{\ell-1}$ submatrix of $B$. Applying the same argument as above,
we find a set $I_\ell\subset I_{\ell-1}$ of size at most $|I_{\ell-1}|\,n^{-\kappa}$ such that
rows $\row_i(B')$, $i\in I_{\ell-1}\setminus I_\ell$ all have expected squared
Euclidean norms less than $n^{-\kappa}\,|z|^2$.
The iterative process stops if $I_\ell$ is empty.
Denote the number of steps by $k$.

Note that the resulting sequence
$$
I_0\supset I_1\supset \dots\supset I_{k}
$$
satisfies
$$
|I_\ell|\leq |I_{\ell-1}|\,n^{-\kappa},\quad 1\leq \ell\leq k,
$$
and $I_{k-1}\neq \emptyset$, whence necessarily
$$
\big(n^{\kappa}\big)^{k-1}\leq n,
$$
implying
$$
k\leq \frac{1}{\kappa}+1.
$$
Choose a permutation $\pi$ so that
$$
\{\pi(1),\pi(2),\dots,\pi(|I_\ell|)\}=I_{\ell},\quad 1\leq \ell\leq k-1,
$$
let $X:=\big(B_{\pi(i),\pi(j)}\big)_{i,j\leq m}$, and let $X_{i,j}$, $1\leq i,j\leq k$,
be the blocks in the $k\times k$ block decomposition of $X$, where
$X_{i,j}$ is $|I_{k-i}\setminus I_{k-i+1}|\times |I_{k-j}\setminus I_{k-j+1}|$
for all admissible $i,j$.
Observe that in this setting for every $1\leq j\leq i\leq k$, all rows of the matrix $X_{i,j}$
have expected squared Euclidean norms at most $n^{-\kappa}\,|z|^2$.
The statement follows.
\end{proof}

\begin{prop}\label{alejhfbwfuvufceutqwfvutfyv}
For every $K\geq 1$, $\kappa\in(0,1]$, and $R\geq 1$ there is
$n_{\text{\tiny\ref{alejhfbwfuvufceutqwfvutfyv}}}\in\N$ depending on $K$, $\kappa$, $R$
with the following property. Let $n\geq n_{\text{\tiny\ref{alejhfbwfuvufceutqwfvutfyv}}}$,
and let $z\in\C\setminus\{0\}$
and the matrix $A=V\odot W$ satisfy the assumptions of Theorem~\ref{qiuyefqiwgfqvuq}.
Condition on any realization of $A$ from $\Event_{\text{\tiny\ref{itfufytcuqtrcutqwxi}}}$. 
Then for every non-empty terminal $A_J-z\,\Id$ of $\mathcal G$, we have
$s_{\min}(A_J-z\,\Id)\geq n^{-\kappa/2}|z|$.
\end{prop}
\begin{proof}
We will assume that $n$ is large.
Condition on any realization of the event $\Event_{\text{\tiny\ref{itfufytcuqtrcutqwxi}}}$,
and consider any non-empty terminal $A_J-z\,\Id$.
Recall that, by construction of $\mathcal G$,
$\Exp\,\|\col_j(A_J)\|_2^2\leq n^{-2\kappa}\,|z|^2$ for all $j\in J$.
Applying Lemma~\ref{auyfiytwvfuyvwfuywvycv}, we get a block matrix
$$
\big((A_J)_{\pi(i),\pi(j)}\big)_{i,j\in J}=:X
=\begin{pmatrix}
X_{1,1} & X_{1,2} & \dots & X_{1,k} \\
X_{2,1} & X_{2,2} & \dots & X_{2,k} \\
\dots & \dots & \dots & \dots \\
X_{k,1} & X_{k,2} & \dots & X_{k,k}
\end{pmatrix},
$$
such that $k\leq \frac{1}{\kappa}+1$ and
for every $1\leq j\leq i\leq k$, all rows of the matrix $X_{i,j}$
have expected squared Euclidean norms at most $n^{-\kappa}\,|z|^2$.
Note further that $\Exp\,\|\col_\ell(X_{i,j})\|_2^2\leq n^{-2\kappa}\,|z|^2$ for every $i,j$.
The definition of $\Event_{\text{\tiny\ref{itfufytcuqtrcutqwxi}}}$ then implies that
\begin{itemize}
\item For every $1\leq j\leq i\leq k$, $\|X_{i,j}\|\leq C_{\text{\tiny\ref{itfufytcuqtrcutqwxi}}}\,n^{-\kappa/2}\,|z|\,\log^{5/2} n$, and
\item For every $1\leq i,j\leq k$,
$\|X_{i,j}\|\leq C_{\text{\tiny\ref{itfufytcuqtrcutqwxi}}}\,R\,|z|\,\log^{5/2} n$.
\end{itemize}
Applying Corollary~\ref{iuygfqwcyvqtrqvieyrc}
with $\varepsilon:=C_{\text{\tiny\ref{itfufytcuqtrcutqwxi}}}\,n^{-\kappa/2}\,\log^{5/2} n$
and with parameter $C_{\text{\tiny\ref{itfufytcuqtrcutqwxi}}}\,R\,\log^{5/2} n$ instead of $R$,
we get
$$
s_{\min}(A_J-z\,\Id)=s_{\min}(X-z\,\Id)\geq \varepsilon\,k\,|z|.
$$
The result follows.
\end{proof}

\subsection{Proof of Theorem~\ref{qiuyefqiwgfqvuq}}

For any structure
$\mathcal S=(\mathcal P,(r_\ell)_{\ell=0}^d)\in\mathcal M$,
where $r_0$ is the largest integer such that
$\exp\big(-\frac{n^{1+3\kappa}(\sigma^*)^2}{|z|^2}\big)
\leq t_{r_0}$,
we define an auxiliary event
$$
\Event_{\mathcal S}:=
\big\{
\mbox{$\mathcal S=(\mathcal P,(r_\ell)_{\ell=0}^d)$ is $A$--compatible and
satisfies properties enlisted in Proposition~\ref{aiytevutvcutvirvuveo}}
\big\}.
$$

Fix any $\mathcal S\in\mathcal M$ with the above choice of $r_0$,
such that $\Prob(\Event_{\mathcal S}\cap \Event_{\text{\tiny\ref{itfufytcuqtrcutqwxi}}})
\neq 0$. Note that this assumption, combined with Proposition~\ref{alejhfbwfuvufceutqwfvutfyv}
and the inequality $n^{-\kappa/2}>t_{p_0}$, implies that
\begin{itemize}
\item Either $\mathcal P[d]$ is an empty terminal and $r_d= -1$,
\item Or $\mathcal P[d]$ is a non-empty terminal and $r_d< p_0$.
\end{itemize}
The latter, in turn, implies that everywhere on $\Event_{\mathcal S}$,
for all $1\leq \ell\leq d$ such that $r_{\ell}< r_{\ell-1}-\tilde p_0$, the edge connecting
$\mathcal P[\ell-1]$ to $\mathcal P[\ell]$ has a non-empty label $j_\ell$, and
$$
\Bigg|\frac{\langle\normal_{J_{\ell-1},j_{\ell-1}},\col_{j_{\ell-1}}(A_{J_{\ell-1}})\rangle
+\langle\normal_{J_{\ell-1},j_{\ell-1}},\col_{j_{\ell-1}}(-z\,\Id)\rangle}
{\sqrt{\sum_{k\in J_{\ell-1}}|(\normal_{J_{\ell-1},j_{\ell-1}})_k|^2\;
|(\col_{j_{\ell-1}}(V_{J_{\ell-1}}))_k|^2}}\Bigg|\leq t_{r_{\ell-1}-r_\ell-\tilde p_0},
$$
where $\mathcal P[\ell-1]=A_{J_{\ell-1}}-z\,\Id$. 
At this point, we implement ``telescopic'' conditioning.
Let $\ell_1,\ell_2,\dots,\ell_h$ be the increasing sequence of all indices $\ell\in\{1,\dots,d\}$
with $r_{\ell}< r_{\ell-1}-\tilde p_0$ (we note that there must be at least one such index,
in view of the upper bound $d\leq \lceil 2n(\sigma^*)^2/L^2\rceil$, the definition of $r_0$, and the above
condition that $r_d<p_0$). Define for every $q=1,\dots,h$,
$$
\Event_{\ell_q}:=
\Bigg\{\Bigg|\frac{\langle\normal_{J_{\ell_q-1},j_{\ell_q-1}},\col_{j_{\ell_q-1}}(A_{J_{\ell_q-1}})\rangle
+\langle\normal_{J_{\ell_q-1},j_{\ell_q-1}},\col_{j_{\ell_q-1}}(-z\,\Id)\rangle}
{\sqrt{\sum_{k\in J_{\ell_q-1}}|(\normal_{J_{\ell_q-1},j_{\ell_q-1}})_k|^2\;
|(\col_{j_{\ell_q-1}}(V_{J_{\ell_q-1}}))_k|^2}}\Bigg|\leq t_{r_{\ell_q-1}-r_{\ell_q}-\tilde p_0}\Bigg\},
$$
and observe that each $\Event_{\ell_q}$ is measurable with respect to the sigma--field
generated by $A_{J_{\ell_q-1}}$, and that
$$
\Event_{\mathcal S}\subset \bigcap\limits_{q=1}^h \Event_{\ell_q}.
$$
We can write
\begin{equation}\label{akfviytrefuyeftvwiy}
\Prob(\Event_{\mathcal S})\leq\Prob\big(\Event_{\ell_h}\big)\,
\prod_{q=h-1}^1 \Prob\big(\Event_{\ell_q}\;|\;
\Event_{\ell_{q+1}}\cap \dots\cap \Event_{\ell_h}\big).
\end{equation}
Lemma~\ref{auevfiwuyfviygvcikygviygv} implies that the 
probability of $\Event_{\ell_h}$ is bounded above by $C\,t_{r_{\ell_h-1}-r_{\ell_h}-\tilde p_0}$,
and that for every $1\leq q\leq h-1$, given any realization of $\mathcal P[\ell_q]=A_{J_{\ell_q}}-z\,\Id$,
the conditional probability of 
$\Event_{\ell_q}$
is bounded above by $C\,t_{r_{\ell_q-1}-r_{\ell_q}-\tilde p_0}$ for some $C\geq 1$ depending
only on $\rho_0$. Since the intersection $\Event_{\ell_{q+1}}\cap \dots\cap \Event_{\ell_h}$
is $\mathcal P[\ell_q]$--measurable, we get from \eqref{akfviytrefuyeftvwiy},
$$
\Prob(\Event_{\mathcal S})
\leq \prod\limits_{q=1}^h\big(C\,t_{r_{\ell_q-1}-r_{\ell_q}-\tilde p_0}\big)
\leq C^d\,\prod_{\ell=1}^d t_{\max(0,r_{\ell-1}-r_\ell-\tilde p_0)}
=C^d\,t_{\,\sum_{\ell=1}^d\max(0,r_{\ell-1}-r_\ell-\tilde p_0)},
$$
where $r_d<p_0$.
Note that
$$
\sum_{\ell=1}^d\max(0,r_{\ell-1}-r_\ell-\tilde p_0)\geq r_0-r_d-d\,\tilde p_0\geq r_0-(d+1)\,\tilde p_0,
$$
and therefore
$$
\Prob(\Event_{\mathcal S})\leq C^d\,2^{(d+1)\,\tilde p_0-r_0}.
$$

\medskip

Combining the cases considered above we get that
for every $\mathcal S=(\mathcal P,(r_\ell)_{\ell=0}^d)\in\mathcal M$
with $r_0:=\big\lfloor
\log_2\exp\big(\frac{n^{1+3\kappa}(\sigma^*)^2}{|z|^2}\big)\big\rfloor$,
the probability of
$\Event_{\mathcal S}\cap\Event_{\text{\tiny\ref{itfufytcuqtrcutqwxi}}}$
is bounded by $C^d\,2^{(d+1)\,\tilde p_0-r_0}$.
Taking the union bound over all structures $\mathcal S\in\mathcal M$ with the
above choice of $r_0$ and applying Lemma~\ref{akgvutrcytqrcwxtqwevuyc} and Proposition~\ref{aiytevutvcutvirvuveo}, we
obtain
\begin{align*}
\Prob&\bigg\{s_{\min}(A-z\,\Id)\leq |z|\,\exp\bigg(-\frac{n^{1+3\kappa}(\sigma^*)^2}{|z|^2}\bigg)\bigg\}\\
&\leq \Prob(\Event_{\text{\tiny\ref{avfiahcgiaygvciywafiyg}}}^c)
+\Prob(\Event_{\text{\tiny\ref{itfufytcuqtrcutqwxi}}}^c)
+\big(n(n+2)\big)^{\lceil 2n(\sigma^*)^2/L^2\rceil}\cdot C^{\lceil 2n(\sigma^*)^2/L^2\rceil}
\,2^{(\lceil 2n(\sigma^*)^2/L^2\rceil+1)\,\tilde p_0-r_0}.
\end{align*}
The required estimate follows by our choice of parameters and
application of Lemma~\ref{avfiahcgiaygvciywafiyg}
and Proposition~\ref{itfufytcuqtrcutqwxi}.

\section{Applications to the circular law}\label{fiauftucytvuyrbiuehbiu}

Here, we discuss Theorem~\ref{alryywtvcutwercuqyt}
in context of convergence of spectral distributions
to the uniform measure on the unit disc.
In what follows, we assume that for each $n$,
$A_n$ is an $n\times n$ matrix with mutually independent centered
real Gaussian entries, and denote
$$
\sigma_n^*:=\max\limits_{i,j\leq n}\sqrt{\Exp\,|(A_n)_{i,j}|^2},\quad
\sigma_n:=\max\bigg(\max\limits_{j\leq n}\sqrt{\Exp\,\|\col_j(A_n)\|_2^2},
\max\limits_{i\leq n}\sqrt{\Exp\,\|\row_i(A_n)\|_2^2}\bigg).
$$
We recall that $\mu_n$ denotes the empirical spectral distribution of $A_n$.
We further let $G_n$ be an $n\times n$ matrix with
i.i.d real Gaussian entries of zero mean and variance $1/n$.
For each $z\in\C$ define random probability measures
$$
\nu_{A_n}(z):=\frac{1}{n}\sum_{i=1}^n \delta_{s_i^2(A_n-z\,\Id)};\quad
\nu_{G_n}(z):=\frac{1}{n}\sum_{i=1}^n \delta_{s_i^2(G_n-z\,\Id)}.
$$

The next theorem is a version of the {\it replacement principle} from \cite{TV2010}
specialized to our setting.
\begin{theorem}[{Replacement principle, \cite{TV2010}}]
Assume that
\begin{itemize}
\item The expression $\frac{1}{n}\|A_n\|_{HS}^2$
is bounded in probability;
\item For almost all complex numbers $z\in\C$,
$$
\frac{1}{n}\log\big|\det\big(A_n-z\,\Id\big)\big|
-\frac{1}{n}\log\big|\det\big(G_n-z\,\Id\big)\big|
$$
converges in probability to zero. 
\end{itemize}
Then the empirical spectral distributions
$(\mu_n)_{n=1}^\infty$ converge weakly in probability to the uniform
measure on the unit disc of the complex plane.
\end{theorem}

\medskip

Theorem~\ref{alryywtvcutwercuqyt} together with the replacement principle imply
\begin{cor}[A sufficient condition for the circular law]\label{aofiffaifuhbvjavbajv}
Fix any constants $C> 0$ and $\varepsilon\in(0,1]$, and 
assume that
\begin{equation}\label{akhvcutwcruqtweufyqw}
n^\varepsilon\sigma_n^*\leq\sigma_n\quad\mbox{and}\quad
\sigma_n\leq C\quad\mbox{for all large $n$.}
\end{equation}
Assume further that for almost every $z\in\C$ the sequence of numbers
$$
n^{\varepsilon}\,
\Big(\frac{\sqrt{n}\,\sigma_n^*}{\sigma_n}\Big)^2\,\log(n)\;\sup\limits_{x\geq 0}\big|\nu_{A_n}((-\infty,x])
-\nu_{G_n}((-\infty,x])\big|,\quad n\geq 1,
$$
converges to zero in probability.
Then the sequence $(\mu_n)_{n=1}^\infty$ of empirical spectral distributions
of matrices $A_n$
converges weakly in probability to the uniform measure on the unit disc of the complex plane.
\end{cor}
\begin{proof}
By Theorem~\ref{alryywtvcutwercuqyt}, for every non-zero $z\in\C$
and all large $n$,
$$
\Prob\bigg\{
s_{\min}(A_n-z\,\Id)\geq |z|\,\exp\bigg(-n^{\varepsilon}\,\Big(\frac{\sqrt{n}\,\sigma_n^*}{\sigma_n}\Big)^2\bigg)
\bigg\}\geq 1-\frac{1}{n},
$$
and hence
$$
\sum_{i:\, s_i(A_n-z\,\Id)\geq |z|\,\exp\big(-n^{\varepsilon}\,
\big(\frac{\sqrt{n}\,\sigma_n^*}{\sigma_n}\big)^2\big)}\log s_i(A_n-z\,\Id)
=\log\big|\det\big(A_n-z\,\Id\big)\big|
$$
with probability $1-o(1)$.
Standard estimates on the smallest singular value of shifted Gaussian matrices
(see \cite{ICM} and references therein)
imply that, similarly,
$$
\sum_{i:\, s_i(G_n-z\,\Id)\geq |z|\,\exp\big(-n^{\varepsilon}\,
\big(\frac{\sqrt{n}\,\sigma_n^*}{\sigma_n}\big)^2\big)}\log s_i\big(G_n-z\,\Id\big)
=\log\big|\det\big(G_n-z\,\Id\big)\big|
$$
with probability $1-o(1)$.
Combining these bounds with \cite[Lemma~4.3]{JJLO}, we get that
with probability $1-o(1)$,
\begin{align*}
&\bigg|\frac{1}{n}\log\big|\det\big(A_n-z\,\Id\big)\big|
-\frac{1}{n}\log\big|\det\big(G_n-z\,\Id\big)\big|\bigg|\\
&\hspace{2cm}\leq C'\,n^{\varepsilon}\,
\Big(\frac{\sqrt{n}\,\sigma_n^*}{\sigma_n}\Big)^2\,\log(n)\;\sup\limits_{x\geq 0}\big|\nu_{A_n}((-\infty,x])
-\nu_{G_n}((-\infty,x])\big|,
\end{align*}
where $C'>0$ is a universal constant.

The proof is accomplished by an application of the replacement principle
(note that, in view of Bernstein's inequality,
$\frac{1}{n}\|A_n\|_{HS}^2$ is bounded from above in probability).
\end{proof}

Corollary~\ref{aofiffaifuhbvjavbajv} establishes the limiting circular law whenever
the sequence
$$
\sup\limits_{x\geq 0}\big|\nu_{A_n}((-\infty,x])
-\nu_{G_n}((-\infty,x])\big|,\quad n\geq 1,
$$
decays to zero faster than
$$
\bigg(n^{\varepsilon}\,
\Big(\frac{\sqrt{n}\,\sigma_n^*}{\sigma_n}\Big)^2\,\log(n)\bigg)^{-1},
$$
for some fixed $\varepsilon>0$.
We shall apply the corollary in the setting of non-Hermitian periodic band matrices.
Let us recall the model. For every $n$ we let $w_n\leq n/2$ be a positive integer
(the bandwidth), and let $B_n$ be an $n\times n$
matrix with mutually independent entries where
the $(i,j)$--th entry is a standard real
Gaussian if and only if $(i-j) \mod n\leq w_n$ or $(j-i) \mod n\leq w_n$,
and all other entries are zeros.
Convergence of the Stieltjes transform
of the spectrum of $\frac{1}{2w_n+1}(B_n-z\,\Id)(B_n-z\,\Id)^*$ was verified in \cite{JS}.
Following the argument of \cite{JS},
asymptotic properties of the singular spectrum
of {\it block-band} matrices were studied \cite{JJLO} as a means to derive a limiting law for the
empirical spectral measure.
It can be verified by going through the proof in \cite{JJLO} that the 
following estimate is valid in the periodic band matrix setting as well.
\begin{theorem}[{Essentially proved in \cite{JJLO}; see \cite[Lemma~4.4]{JJLO}}]
There is a universal constant $c>0$ with the following property.
Let $(w_n)_{n\geq 1}$ be a sequence of integers where each $w_n$ satisfies
$$
c^{-1}\,n^{32/33}\log n\leq w_n< cn.
$$
Then for every fixed $z\in\C$ we have
$$
\sup\limits_{x\geq 0}\big|\nu_{\frac{1}{\sqrt{2w_n+1}}B_n}((-\infty,x])
-\nu_{G_n}((-\infty,x])\big|
\leq C'\,\bigg(\frac{n\log n}{w_n^2}\bigg)^{1/31}
$$
with probability $1-o(1)$,
where $C'>0$ may only depend on $z$.
\end{theorem}
As a combination of the last theorem and Corollary~\ref{aofiffaifuhbvjavbajv}, we obtain
Corollary~\ref{ajhfbaifhbijvnjgvaigvi}.


\begin{thebibliography}{99}

\bibitem{AEK2018}
{
J. Alt, L. Erd\H{o}s\ and\ T. Kr\"{u}ger, Local inhomogeneous circular law, Ann. Appl. Probab. {\bf 28} (2018), no.~1, 148--203. 
}

\bibitem{AEK2021}
{
J. Alt, L. Erd\H{o}s\ and\ T. Kr\"{u}ger, Spectral radius of random matrices with independent entries, Probab. Math. Phys. {\bf 2} (2021), no.~2, 221--280.
}

\bibitem{B1997}
{
Z. D. Bai, Circular law, Ann. Probab. 25 (1997), no. 1, 494--529.
}

\bibitem{BS2010}
{
Z.~D. Bai\ and\ J.~W. Silverstein,
{\it Spectral analysis of large dimensional random matrices}, second edition, Springer Series in Statistics, Springer, New York, 2010.
}

\bibitem{BBvH}
{
A.~S. Bandeira, M. Boedihardjo and\ R. van~Handel, 
Matrix concentration inequalities and free probability, Invent. Math., to appear.
}

\bibitem{BvH}
{
A.~S. Bandeira\ and\ R. van~Handel, Sharp nonasymptotic bounds on the norm of random matrices with independent entries, Ann. Probab. {\bf 44} (2016), no.~4, 2479--2506.
}

\bibitem{BCZ}
{
A. Basak, N.~A. Cook\ and\ O. Zeitouni, Circular law for the sum of random permutation matrices, Electron. J. Probab. {\bf 23} (2018), Paper No. 33, 51 pp.
}

\bibitem{BR2019}
{
A. Basak, M. Rudelson, The circular law for sparse non-Hermitian matrices,
Ann. Probab. 47 (2019), no.~4, 2359--2416.
}

\bibitem{BC2012}
{
C. Bordenave, J. Chafa\"i, Around the circular law. Probability Surveys 9(0), 1--89, 2012.
}

\bibitem{BvH22}
{
T.~Brailovskaya, R. van~Handel,
Universality and sharp matrix concentration inequalities, arXiv:2201.05142
}

\bibitem{CookInv}
{
N.~A. Cook, Lower bounds for the smallest singular value of structured random matrices,
Ann. Probab. {\bf 46} (2018), no.~6, 3442--3500.
}

\bibitem{Cook2016}
{
N. A. Cook. Spectral properties of non-Hermitian random matrices.
PhD thesis, University of California, Los Angeles, 2016.
}

\bibitem{Cook2019}
{
N. Cook,
The circular law for random regular digraphs.
Ann. Inst. Henri Poincar\'e Probab. Stat. 55 (2019), no. 4, 2111--2167.
}

\bibitem{CHNR1}
{
N. Cook,  W. Hachem, J. Najim, D. Renfrew,
Non-Hermitian random matrices with a variance profile (I): deterministic equivalents and limiting ESDs.
Electron. J. Probab. 23 (2018), Paper No. 110, 61 pp.
}

\bibitem{CHNR2}
{
N. Cook,  W. Hachem, J. Najim, D. Renfrew,
Non-Hermitian random matrices with a variance profile (II):
Properties and examples, J. Theoret. Probab. {\bf 35} (2022), no.~4, 2343--2382.
}

\bibitem{Ginibre}
{
J. Ginibre, Statistical ensembles of complex, quaternion, and real matrices,
J. Mathematical Phys. 6 (1965), 440--449.
}

\bibitem{Girko}
{
V. L. Girko, The circular law, Teor. Veroyatnost. i Primenen. 29 (1984), no. 4, 669--679.
}

\bibitem{GT2010}
{
F. G\"otze and A. Tikhomirov, The circular law for random matrices,
Ann. Probab. 38 (2010), no. 4,
1444--1491.
}

\bibitem{HT}
{
H. Huang, K. Tikhomirov, On dimension-dependent concentration for convex Lipschitz functions in product spaces, Electronic Journal of Probability, to appear.
}

\bibitem{JJLO}
{
V. Jain, I. Jana, K. Luh, S. O’Rourke,
Circular law for random block band matrices with genuinely sublinear bandwidth,
J. Math. Phys. 62, 083306 (2021)
}

\bibitem{JS21}
{
V. Jain\ and\ S. Silwal, A note on the universality of ESDs of
inhomogeneous random matrices, ALEA Lat. Am. J. Probab. Math. Stat. {\bf 18} (2021), no.~2, 1047--1059.
}

\bibitem{JS}
{
I. Jana\ and\ A.~B. Soshnikov, Distribution of singular values of random band matrices; Marchenko--Pastur law and more, J. Stat. Phys. {\bf 168} (2017), no.~5, 964--985.
}

\bibitem{Latala}
{
R. Lata\l a, Some estimates of norms of random matrices,
Proc. Amer. Math. Soc. {\bf 133} (2005), no.~5, 1273--1282. MR2111932
}

\bibitem{LvHY}
{
R. Lata\l a, R. van~Handel\ and\ P. Youssef,
The dimension-free structure of nonhomogeneous random matrices,
Invent. Math. {\bf 214} (2018), no.~3, 1031--1080.
}

\bibitem{LLTTY}
{
A.E.Litvak, A.Lytova, K.Tikhomirov, N.Tomczak-Jaegermann, P.Youssef,
Circular law for sparse random regular digraphs, J. Eur. Math. Soc. (JEMS) 23 (2021), no. 2, 467--501.
}

\bibitem{Galya}
{
G.~V. Livshyts, The smallest singular value of heavy-tailed not necessarily i.i.d. random matrices via random rounding, J. Anal. Math. {\bf 145} (2021), no.~1, 257--306.
}

\bibitem{LTV}
{
G.~V. Livshyts, K.~E. Tikhomirov\ and\ R. Vershynin, The smallest singular value of inhomogeneous square random matrices, Ann. Probab. {\bf 49} (2021), no.~3, 1286--1309.
}

\bibitem{Molchanov}
{
I.~S. Molchanov, {\it Theory of random sets}, second edition,
Probability Theory and Stochastic Modelling, 87, Springer, London, 2017.
}

\bibitem{Olshevsky}
{
Structured matrices in mathematics, computer science, and engineering. I.
Proceedings of the AMS-IMS-SIAM Joint Summer Research Conference on Structured Matrices in Operator Theory, Numerical Analysis, Control, Signal and Image Processing held at the University of Colorado, Boulder, CO, June 27--July 1, 1999. Edited by V. Olshevsky,
Contemp. Math., 280
American Mathematical Society, Providence, RI, 2001. xiv+327 pp.
}

\bibitem{PZ2010}
{
G. Pan and W. Zhou, Circular law, extreme singular values and potential theory,
J. Multivariate Anal. 101 (2010), no. 3, 645--656.
}

\bibitem{RS2013}
{
S. Riemer\ and\ C. Sch\"{u}tt, On the expectation of the norm of random matrices with non-identically distributed entries, Electron. J. Probab. {\bf 18} (2013), no. 29, 13 pp.
}

\bibitem{Rogozin1987}
{
Rogozin, B. A.
An estimate for the maximum of the convolution of bounded densities.
Teor. Veroyatnost. i Primenen.32 (1987), no.1, 53--61.
}

\bibitem{RT2019}
{
M. Rudelson, K. Tikhomirov, The sparse circular law under minimal assumptions,
Geom. Funct. Anal. 29 (2019), no. 2, 561--637.
}

\bibitem{RV08}
{
M. Rudelson\ and\ R. Vershynin,
The Littlewood-Offord problem and invertibility of random matrices,
Adv. Math. {\bf 218} (2008), no.~2, 600--633.
}

\bibitem{RV15}
{
M. Rudelson, R. Vershynin,
Small ball probabilities for linear images of high-dimensional distributions.
Int. Math. Res. Not. IMRN (2015), no.19, 9594--9617.
}

\bibitem{RZ}
{
M. Rudelson, O. Zeitouni, Singular values of Gaussian matrices and permanent estimators, Random Structures Algorithms 48 (2016), no. 1, 183--212.
}

\bibitem{TV2008}
{
T. Tao and V. Vu, Random matrices: the circular law, Commun. Contemp. Math. 10 (2008), no. 2, 261--307.
}

\bibitem{TV2010}
{
T. Tao and V. Vu, Random matrices: universality of ESDs and the circular law, Ann. Probab. 38 (2010),
no. 5, 2023--2065.
}

\bibitem{ICM}
{
K.Tikhomirov, Quantitative invertibility of non-Hermitian random matrices, to appear in the ICM 2022 proceedings.
}

\bibitem{vH17}
{
R. van~Handel, On the spectral norm of Gaussian random matrices.
Trans. Amer. Math. Soc. 369 (2017), no.11, 8161--8178.
}

\bibitem{VershyninBook}
{
R. Vershynin, {\it High-dimensional probability}, Cambridge Series in Statistical and Probabilistic Mathematics, 47, Cambridge Univ. Press, Cambridge, 2018.
}

\bibitem{W2012}
{
P. M. Wood, Universality and the circular law for sparse random matrices,
Ann. Appl. Probab. 22 (2012), no. 3, 1266--1300.
}

\end{thebibliography}
\end{document}